\documentclass[preprint,12pt]{elsarticle}
\usepackage{amsmath}
\usepackage{amsthm}
\usepackage{amsfonts}
\usepackage{amssymb}
\usepackage{amsbsy}

\usepackage{float}

\usepackage[pagebackref=true, bookmarksopen=true,colorlinks=true, linkcolor=red,citecolor=blue]{hyperref}

\usepackage[capitalise]{cleveref}  

\usepackage{tikz}
\usetikzlibrary{decorations.pathmorphing}

\usepackage{subcaption}

\newtheorem{theorem}{Theorem}[section]
\newtheorem{lemma}[theorem]{Lemma}
\newtheorem{corollary}[theorem]{Corollary}
\begin{document}
	
	\begin{abstract}
We introduce two new types of graph configurations, the Jflower and the Jposy, which generalize the classical flower and posy configurations of Edmonds, Sterboul, and Deming in the context of maximum matchings. These generalized configurations allow greater flexibility in characterizing non-K\H{o}nig--Egerv\'{a}ry graphs and provide new tools for studying matching-theoretic properties.

Our main result shows that the sets of vertices covered by classical configurations (flowers and posies), restricted configurations (Tposies), and generalized configurations (Jflowers and Jposies) coincide. This equivalence yields a unified characterization of what we call \emph{Sterboul–Deming graphs}—graphs in which every vertex belongs to some configuration relative to an appropriate maximum matching.
	\end{abstract}


	\begin{frontmatter}
		
		\title{Generalized Edmonds-Sterboul-Deming configurations.\\ Part 1: Sterboul-Deming graphs}
		
		\author[pan,daj]{Daniel A. Jaume}
		\ead{djaume@unsl.edu.ar}
		
		\author[pan]{Cristian Panelo}
		\ead{crpanelo@unsl.edu.ar}
		\author[pan,daj]{Kevin Pereyra}
		\ead{kdpereyra@unsl.edu.ar}
		\address[pan]{Universidad Nacional de San Luis, Argentina.}
		\address[daj]{IMASL-CONICET, Argentina.}


		\date{Received: date / Accepted: date}

	\end{frontmatter}
	%
	%
	%
	\section{Introduction}
	
	The interplay between matchings and vertex covers in graphs has been a central theme in combinatorial optimization. A graph $G$ is called \emph{K\H{o}nig--Egerv\'{a}ry} if its matching number equals its vertex cover number, or equivalently, if $\alpha(G) + \mu(G) = |G|$, where $\alpha(G)$ denotes the independence number and $\mu(G)$ the matching number.
	
	The structural characterization of K\H{o}nig--Egerv\'{a}ry graphs was achieved through the introduction of specific subgraph configurations. 
	
	In 1965, Edmonds~\cite{edmonds1965paths} introduced the notions of \emph{blossom} and \emph{flower} as part of his polynomial-time algorithm for finding maximum matchings in arbitrary graphs. Later, Sterboul~\cite{sterboul1979characterization} and, independently, Deming~\cite{deming1979independence} introduced a third configuration called \emph{posy} (referred to as blossom pairs by Deming). These configurations characterize non-K\H{o}nig--Egerv\'{a}ry graphs: A graph is K\H{o}nig--Egerv\'{a}ry if and only if it contains no flowers or posies relative to any maximum matching.
	
Following Korach et al.~\cite{korach2006subgraph}, we use the term \emph{configuration}  to denote a subgraph together with an associated maximum matching. This terminology highlights the dependence of such structures on both the graph topology and the chosen matching.

Recent work by Jaume and Molina~\cite{jaume2020fpke} introduced a dual perspective by defining \emph{FP-graphs}—graphs in which every vertex belongs to some configuration (either a flower or a posy). Such graphs can be viewed as the structural counterparts to K\H{o}nig–Egerv\'{a}ry graphs. 

In this work, we adopt a historical perspective and propose the name \emph{Sterboul–Deming graphs} for this class, in honor of the two researchers who independently discovered posies and established their relationship with non-K\H{o}nig–Egerv\'{a}ry graphs.
	
	The classical configurations, though sufficient for characterizing non--K\H{o}nig–Egerv\'{a}ry graphs, prove restrictive when analyzing more complex structural properties. For instance, flowers are constructed from alternating paths (which, by definition, cannot revisit vertices), while posies impose strict conditions on path lengths and disjointness. These constraints, while essential for their original applications, limit the flexibility required for developing more general decomposition theories.
	
	To address this limitation, we introduce two generalized configurations, \emph{Jflowers} and \emph{Jposies}, that retain the essential vertex-covering properties of their classical counterparts while offering greater structural flexibility.
	
	Our main theoretical contribution, presented in \cref{teo_main1}, is to establish that these generalized configurations are equivalent to the classical ones with respect to vertex coverage. This result is non-trivial, as its proof requires a careful analysis of how a generalized configuration can be transformed into a classical one while preserving the set of covered vertices. The techniques employed involve matching modifications, walk simplification arguments, and induction on the graph's structure.

	The generalized configurations introduced here serve as building blocks for a comprehensive graph decomposition theory. In subsequent papers, we will show that every graph admits an \emph{SD-KE decomposition}—a partition into Sterboul-Deming components and K\H{o}nig--Egerv\'{a}ry components. This decomposition is additive with respect to the matching number and the independence number, and we conjecture that is also multiplicative with respect to determinant in graphs.
	
	The flexibility of Jflowers and Jposies is crucial for these applications. Classical configurations, with their rigid path requirements, do not always behave well under graph concatenation. In contrast, the walk-based generalization permits more natural decomposition arguments while preserving the fundamental vertex-covering properties that make configurations valuable tools for characterizing graph classes.

	The techniques developed here—particularly the systematic analysis of alternating walks and their simplifications—may be of independent interest for other problems in matching theory and graph structure.
	
	This work is organized as follows: Section~2 establishes notation and reviews preliminary concepts from matching theory. Section~3 introduces Jflower and analyzes the relationship with the classical flower configuration. Section~4 defines Jposy and the more basic configuration od Tposy. Section~5 proves that Jposies mark the same vertices as classical posies, which forms the technical heart of our equivalence theorem. Section~6 shows that Jflower vertices can be covered by classical configurations. Finally, Section~7 synthesizes these results to prove our main theorem.

	%
	%
	%
	
	%
	%
	%
	\section{Notation and preliminaries}
	For all graph-theoretic notions not defined here, we refer the reader to \cite{diestel2000graph} or \cite{chartrand2024graphs}. Throughout this work, all graphs are finite, undirected, and simple (i.e., without loops or multiple edges).
	
	Let $G$ be a graph. The vertex set and edge set of $G$ are denoted by $V(G)$ and $E(G)$, respectively. The order of $G$ is $|G| := |V(G)|$. If $\{u,v\} \in E(G)$, we write $uv$ instead of $\{u,v\}$.
	
	If $H$ is a subgraph of $G$, we write $H \subseteq G$. For $X \subseteq V(G)$, we denote by $G[X]$ the subgraph of $G$ induced by $X$. If $F \subseteq G$, then $G - F$ is the subgraph of $G$ with vertex set $V(G) \setminus V(F)$ and edge set $E(G) \setminus E(F)$. We also write $G - e$, $G - v$, and $G - H$ for the subgraphs obtained by deleting an edge $e \in E(G)$, a vertex $v \in V(G)$, or a subgraph $H \subseteq G$, respectively.

%
%
	
	A \emph{matching} $M$ in $G$ is a set of pairwise nonadjacent edges. An edge $e \in M$ is called \emph{matched} by $M$, and a vertex $v \in V(G)$ is \emph{matched by $M$} if $v$ is an endpoint of some $e \in M$; otherwise, $v$ is \emph{unmatched by $M$}. A vertex set $H \subseteq V(G)$ is said to be \emph{matchable} if there exists a matching in $G$ that saturates all vertices in $H$. 	A matching defines an involution in $G$: we adopt the convention that $M(v) = u$ if $uv \in M$, and $M(v) = v$ if $v$ is unmatched by $M$. A matching $M$ is \emph{perfect} if every vertex is matched, i.e., $M(v) \ne v$ for all $v \in V(G)$. The \textit{matching number} of $G$, denoted $\mu(G)$, is the maximum size of a matching in $G$. For a subgraph $H$ of $G$, we denote by $M(H)$ the set of edges in $H$ that are matched by $M$. 
	
	A \emph{walk} in $G$ is a sequence of vertices $v_1v_2\cdots v_k$ such that $v_iv_{i+1} \in E(G)$ for all $i$. A \emph{path} in $G$ is walk where the vertices are all different. Given two paths $P=u\dots v$ and $Q = w\dots x$, if $vw\in E(G)$, we write $P, Q:= u\dots v w \dots x$. With $P^{-1}$ we denote the path $v\dots u$. If $G$ is connected and $u,v\in V(G)$, with $uPv$ we denote a path between $u$ and $v$. The \emph{distance} from a vertex $x$ to a vertex $y$ is the length of the shortest path $xPy$ and is denoted $d_G(x,y)$.
	
	Let $M$ be a matching in $G$. A path or a walk is called \emph{alternating} with respect to $M$ if, for each pair of consecutive edges in the path (walk), exactly one of them belongs to $M$. Given an alternating path (walk) $P$, we say that $P$ is: \emph{mm-alternating} if it starts and ends with edges in $M$, \emph{nn-alternating} if it starts and ends with edges not in $M$, \emph{mn-alternating} if it starts with an edge in $M$ and ends with one not in $M$, and \emph{nm-alternating} if it starts with an unmatched edge and ends with a matched edge.
	
	A \emph{cycle} in $G$ is called \emph{even} if it has an even number of edges. An \emph{even alternating cycle} with respect to a matching $M$ is an even cycle whose edges alternate between belonging and not belonging to $M$.

	An \textit{independent set} in a graph is a set of vertices, no two of which are adjacent. The \textit{independence number}, defined as the cardinality of a maximum independent set, is denoted by $\alpha(G)$. 
	
	The \textit{(vertex) covering number}, \(\tau(G)\), is the minimum number of vertices required to intersect every edge of \(G\). Graphs satisfying \(\mu(G) = \tau(G)\) are called \textit{K\H{o}nig--Egerv\'{a}ry graphs} (KE graphs, for short) or are said to \textit{have the K\H{o}nig--Egerv\'{a}ry property}. By Gallai's Theorem, a graph $G$ is a K\H{o}nig-Egerv\'ary graph if and only if $\alpha(G) + \mu(G) = |G|$.
	
	An \emph{even subdivision} of an edge $uv$ is the path $uw_1w_2v$, where $w_1$ and $w_2$ are new vertices. An even subdivision of a graph $G$ is any graph obtained from $G$ by repeatedly applying even subdivisions to its edges.

%
	%
	%
	%
	%
	%
	%
	
	\section{Jflower}
	The Jflowers are a generalizations of flowers where the alternating path from the blossom base to an unsaturated vertex is replaced by an alternating walk, allowing for repeated vertices and more complex routing.

	Edmonds \cite{edmonds1965paths} defines the following configurations for a graph \( G \) relative to a matching \( M \). An \( M \)-\emph{blossom} is an odd cycle of length \( 2k + 1 \) in \( G \) that contains exactly \( k \) edges from \( M \). The \emph{base} of a blossom is the unique vertex that is not matched by \( M \) to another vertex within the blossom.
	
	The graph in \cref{fig:Econfigurations} contains a maximum matching, shown in red. The cycles $0,1,2,0$ and $2,3,4,2$\footnote{When we label the vertices of a graph with integers, we write paths and walks using commas.}, together with the matching, form two blossoms with bases $0$ and $2$, respectively; see \cref{fig:blossoms}.

	\begin{figure}[H]
		\centering
		\def \escala{0.65}
		\def \ancho{0.32}
		\def \amplitud{2.5mm}
		\def \opacidad{0.2}
		\begin{subfigure}{\ancho\textwidth}
			\centering
			\begin{tikzpicture}[scale=\escala]
				
				\draw[yellow, opacity=2*\opacidad, line width=10pt, rounded corners] (3,4) -- (3,6) -- (5,6) -- cycle;
				\draw[green, opacity=0.4, line width=10pt, rounded corners]  (5,6) -- (7,6) -- (5,4)-- cycle;
				
				\node[draw,circle, thick, fill = yellow!40!white,scale=0.6] (0) at (3,4) {\(\mathbf{0}\)};
				\node[draw,circle, thick,scale=0.6] (1) at (3,6) {\(1\)};
				\node[draw,circle, thick,scale=0.6] (2) at (5,6) {\textcolor{black}{\(\mathbf{2}\)}};
				\node[draw,circle, thick,scale=0.6] (3) at (7,6) {\(3\)};
				\node[draw,circle, thick,scale=0.6] (4) at (5,4) {\(4\)};
				
				\foreach \from/\to in { {1/2}, {3/2}, {3/4}, {4/2}, {4/0}, {0/0}, {1/0}, {0/2}} {
					\path[draw, very thick] (\from) -- (\to);}
				
				\foreach \from/\to in { {1/2}, {4/3}} {
					\path[decorate, decoration={snake,amplitude=4, segment length=3mm},draw=red, very thick] (\from) -- (\to);}
			\end{tikzpicture}
			\subcaption{Blossoms}\label{fig:blossoms}
		\end{subfigure}
		\begin{subfigure}{\ancho\textwidth}
			\centering
			\begin{tikzpicture}[scale=\escala]
				\draw[blue, opacity=\opacidad, line width=10pt, rounded corners] (3,4) -- (3,6) -- (5,6);
				\draw[green, opacity=\opacidad, line width=10pt, rounded corners] (5,6) -- (7,6) -- (5,4) -- (5,6);
				
				\node[draw,circle, thick,fill = blue!20!white,scale=0.6] (0) at (3,4) {\(\mathbf{0}\)};
				\node[draw,circle, thick,scale=0.6] (1) at (3,6) {\(1\)};
				\node[draw,circle, thick,scale=0.6] (2) at (5,6) {\(2\)};
				\node[draw,circle, thick,scale=0.6] (3) at (7,6) {\(3\)};
				\node[draw,circle, thick,scale=0.6] (4) at (5,4) {\(4\)};
				
				\foreach \from/\to in { {1/2}, {3/2}, {3/4}, {4/2}, {4/0}, {0/0}, {1/0}, {0/2}} {
					\path[draw, very thick] (\from) -- (\to);}
				
				\foreach \from/\to in { {1/2}, {4/3}} {
					\path[decorate, decoration={snake,amplitude=4, segment length=3mm},draw=red,  very thick] (\from) -- (\to);}
			\end{tikzpicture}
			\subcaption{Stem and blossom}\label{fig:stem}
		\end{subfigure}
		\begin{subfigure}{\ancho\textwidth}
			\centering
			\begin{tikzpicture}[scale=\escala]
				\draw[blue, opacity=\opacidad, line width=10pt, rounded corners] (3,4) -- (3,6) -- (5,6) -- (7,6) -- (5,4) -- (5,6);
				
				\node[draw,circle, thick,fill = blue!20!white,scale=0.6] (0) at (3,4) {\(\mathbf{0}\)};
				\node[draw,circle, thick,scale=0.6] (1) at (3,6) {\(1\)};
				\node[draw,circle, thick,scale=0.6] (2) at (5,6) {\(2\)};
				\node[draw,circle, thick,scale=0.6] (3) at (7,6) {\(3\)};
				\node[draw,circle, thick,scale=0.6] (4) at (5,4) {\(4\)};
				
				\foreach \from/\to in { {1/2}, {3/2}, {3/4}, {4/2}, {4/0}, {0/0}, {1/0}, {0/2}} {
					\path[draw, very thick] (\from) -- (\to);}
				
				\foreach \from/\to in { {1/2}, {4/3}} {
					\path[decorate, decoration={snake,amplitude=4, segment length=3mm},draw=red, very thick] (\from) -- (\to);}
			\end{tikzpicture}
			\subcaption{Flower}\label{fig:flower}
		\end{subfigure}
		\caption{Edmonds configurations}\label{fig:Econfigurations}
	\end{figure}
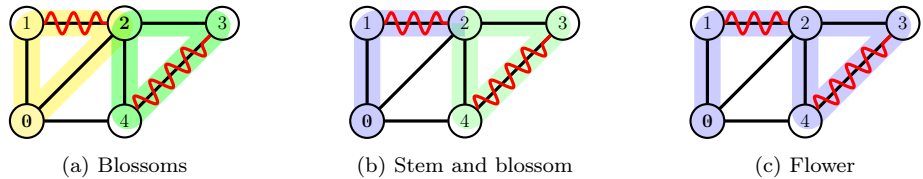

	An \( M \)-\emph{stem} is an even-length \( M \)-alternating path that connects the base of an \( M \)-blossom to a vertex unmatched by \( M \). The base is the only vertex shared between the blossom and the stem. In \cref{fig:stem}, the \( M \)-alternating path \( 0,1,2 \) is the stem associated with the blossom \( 2,3,4,2 \).
	
	An \( M \)-\emph{flower} consists of an \( M \)-blossom together with an associated \( M \)-stem; see \cref{fig:flower}. Note that the stem may be a path of length zero, meaning a blossom can also be a flower. For instance, the blossom \( 0,1,2,0 \) is a flower, whereas the blossom \( 2,3,4,2 \) is not; see \cref{fig:blossoms}.

	 We define the following configuration as a generalization of Edmonds' flower configuration. Given a graph \( G \) and a maximum matching \( M \) of \( G \), an \( M \)-\emph{Jflower} consists of an \( M \)-blossom and an \( M \)-alternating walk (with respect to \( M \)) that starts at the base of the blossom and ends in an \(M\)-unsaturated vertex of \(G\). 
	 
	 In \cref{fig:Jflower}, a maximum matching is shown in red. The marked Jflower is formed by the blossom \( 0,1,2,3,4,0 \), with base \( 0 \), and the \( M \)-alternating walk from vertex \( 0 \) to vertex \( 12 \) given by the sequence
	 \[
	 0,5,9,10,14,13,9,10,14,13,9,10,6,7,8,11,7,6,1,2,12.
	 \]

	\begin{figure}[H]
		\centering
		\def \Vescala{1}
		\def \opacidad{0.2}
		\begin{tikzpicture}[thick,scale=0.6,baseline=(0.base)]
			\node[draw,red,circle,scale=0.6,fill = blue!20!white] (0) at (0,0) {$\mathbf{0}$};
			\node[draw,circle,scale=0.6] (1) at (-2,0) {1};
			\node[draw,circle,scale=0.6] (2) at (-2.7,-2) {2};
			\node[draw,circle,scale=0.6] (3) at (-1,-3) {3};
			\node[draw,circle,scale=0.6] (4) at (0.7,-2) {4};
			\node[draw,circle,scale=0.6] (5) at (2,0) {5};
			\node[draw,circle,scale=0.6] (6) at (2,2+1) {6};
			\node[draw,circle,scale=0.6] (7) at (0,2+1) {7};
			\node[draw,circle,scale=0.6] (8) at (-2,2+1) {8};
			\node[draw,circle,scale=0.6] (9) at (4,0) {9};
			\node[draw,circle,scale=0.6] (10) at (4,2+1) {10};
			\node[draw,circle,scale=0.6] (11) at (-4,1) {11};
			\node[draw,circle,scale=0.6,fill = blue!20!white] (12) at (-4,-3) {$\mathbf{12}$};
			\node[draw,circle,scale=0.6] (13) at (6,0) {13};
			\node[draw,circle,scale=0.6] (14) at (6,2+1) {14};
			\node[draw,circle,scale=0.6] (P) at (2.7,-2) {15};
			\node[draw,circle,scale=0.6] (Q) at (4.7,-2) {16};
			
			\foreach \from/\to in {1/0,0/4,2/3,7/8,5/9,6/10,11/1,2/12,9/13,10/14,4/P,Q/13,11/12,0/8,4/9,1/3,1/7,11/7,6/1} {
				\path[draw, thick] (\from) -- (\to);
			}
			\foreach \from/\to in {0/5,{1/2},3/4,6/7,9/10,8/11,14/13,P/Q} {
				\path [draw, decorate, decoration={snake, segment length=3mm,amplitude=4}, very thick, red] (\from) -- (\to);
			}
			\draw[blue, opacity=\opacidad, line width=10pt, rounded corners] (0) -- (1) -- (2) -- (3) -- (4) -- (0)-- (5)-- (9)-- (10)-- (14)-- (13)-- (9)-- (10)-- (6)-- (7)-- (8)-- (11)-- (7)-- (6)-- (1)--(2)--(12);
		\end{tikzpicture}
		\caption{$M$-Jflower}\label{fig:Jflower}
	\end{figure}
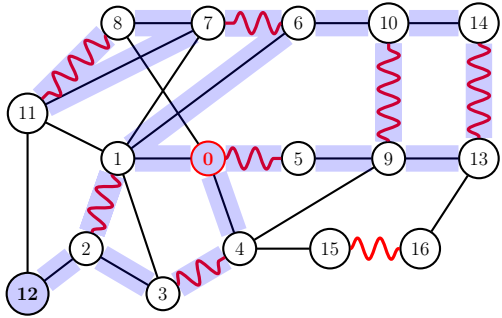

	\section{Jposies}
	
	A Jposy is a generalization of the configuration called a posy, in which the alternating path connecting two blossom bases is replaced by an alternating walk, thereby removing restrictions on self-intersections.

	Sterboul~\cite{sterboul1979characterization} and Deming~\cite{deming1979independence} introduced a configuration called a \emph{posy}. Given a graph $G$ and a maximum matching $M$ of $G$, an $M$-posy is the configuration consisting of two distinct (not necessarily disjoint) $M$-blossoms joined by an odd-length $mm$-alternating path, with respect to $M$, such that the endpoints of the $mm$-alternating path are the bases of the two $M$-blossoms.  
	
	In \cref{fig:posies}, a maximum matching $M$ of a graph $G$ is shown in red. With this matching, $G$ contains many $M$-posies. The one shown in brown in \cref{fig:posies} is formed by the blossom $9,5,4,9$ with base $9$, the blossom $1,2,3,4,5,1$ with base $1$, and the $mm$-alternating path $9,8,5,4,3,2,0,1$.
	
%
%
%
%
%
%
	
	We introduce two configurations related to the notion of a posy, one being a simplification, and the other a generalization.
	
	Given a graph \( G \) and a maximum matching \( M \), an \( M \)-Tposy is defined similarly to a posy, with the additional requirement that the \( M \)-alternating path connecting the bases of the blossoms must be internally disjoint from the blossoms themselves. 
	
	For example, in \cref{fig:Tposy}, the \( M \)-Tposy is formed by the same two blossoms as in the original posy, but the \( M \)-alternating path is \( 9,8,7,6,0,1 \), which is internally disjoint from the blossoms. 
	
	Tposies correspond to even subdivisions of barbells or \(K_{4}\)(see \cite{bonomo2013forbidden}).
	
%
%
%
%

Given a graph \( G \) and a maximum matching \( M \) of \( G \), an \( M \)-Jposy is a configuration consisting of two (not necessarily distinct) \( M \)-blossoms joined by an odd-length \( M \)-alternating walk such that the endpoints of the walk are the bases of the two \( M \)-blossoms.

The following observation is important.
	
\begin{lemma}\label[lemma]{lem:JObs}
		Let \(G\) be a graph and \(M\) be a maximum matching of \(G\). If \(H\) is an \(M\)-Tposy of \(G\) and \(v \in V(H)\), then there exists a non-trivial \(M\)-\(mm\)-alternating path in \(H\) from \(v\) to the base of one the blossoms of \(H\).
\end{lemma}
			
	\begin{figure}[H]
		\centering
		\begin{subfigure}{0.32\textwidth}
			\centering
			\def \Vescala{1}
			\def \opacidad{0.2}
			\begin{tikzpicture}[thick,scale=0.65,baseline=(0.base)]
				\node[draw,circle, scale=0.6] (0) at (-2,1) {0};
				\node[draw,red,circle,scale=0.6,fill = brown!40!white] (1) at (-4,2) {$\mathbf{1}$};
				\node[draw,circle,scale=0.6] (2) at (-5,1) {2};
				\node[draw,circle,scale=0.6] (3) at (-5,-1) {3};
				\node[draw,circle,scale=0.6] (4) at (-4,-2) {4};
				\node[draw,circle,scale=0.6] (5) at (-3,0) {5};
				\node[draw,circle,scale=0.6] (6) at (-1,3) {6};
				\node[draw,circle,scale=0.6] (7) at (0,2) {7};
				\node[draw,circle,scale=0.6] (8) at (0,0) {8};
				\node[draw,red,circle,scale=0.6,fill = brown!40!white] (9) at (-2,-1) {{$\mathbf{9}$}};
				
				\foreach \from/\to in {1/2,1/5,3/4,0/6,7/8,9/5,9/4,0/5,5/8,2/0} {
					\path[draw] (\from) -- (\to);
				}
				\foreach \from/\to in {2/3,4/5,1/0,6/7,8/9} {
					\path [draw, decorate, decoration={snake, segment length=3mm, amplitude=4}, very thick, red] (\from) -- (\to);
				}
				\draw[brown, opacity=2*\opacidad, line width=10pt, rounded corners] 
				(9) -- (5) -- (4) -- (9) -- (8) -- (5)-- (4)-- (3)-- (2)-- (0)-- (1)-- (2)-- (3)-- (4)-- (5)-- (1);
			\end{tikzpicture}
			\caption{\(M\)-posy}
			\label{fig:posies}
		\end{subfigure}
		\hfill
		\begin{subfigure}{0.32\textwidth}
			\centering
			\def \Vescala{1}
			\def \opacidad{0.2}
			\begin{tikzpicture}[thick,scale=0.65,baseline=(0.base)]
				\node[draw,circle, scale=0.6] (0) at (-2,1) {0};
				\node[draw,red,circle,scale=0.6,fill = brown!40!white] (1) at (-4,2) {$\mathbf{1}$};
				\node[draw,circle,scale=0.6] (2) at (-5,1) {2};
				\node[draw,circle,scale=0.6] (3) at (-5,-1) {3};
				\node[draw,circle,scale=0.6] (4) at (-4,-2) {4};
				\node[draw,circle,scale=0.6] (5) at (-3,0) {5};
				\node[draw,circle,scale=0.6] (6) at (-1,3) {6};
				\node[draw,circle,scale=0.6] (7) at (0,2) {7};
				\node[draw,circle,scale=0.6] (8) at (0,0) {8};
				\node[draw,red,circle,scale=0.6,fill = brown!40!white] (9) at (-2,-1) {{$\mathbf{9}$}};
				
				\foreach \from/\to in {1/2,1/5,3/4,0/6,7/8,9/5,9/4,0/5,5/8,2/0} {
					\path[draw] (\from) -- (\to);
				}
				\foreach \from/\to in {2/3,4/5,1/0,6/7,8/9} {
					\path [draw, decorate, decoration={snake, segment length=3mm, amplitude=4}, very thick, red] (\from) -- (\to);
				}
				\draw[orange, opacity=2*\opacidad, line width=10pt, rounded corners] 
				(9) -- (5) -- (4) -- (9) -- (8) -- (7)-- (6)-- (0)-- (1)-- (2)-- (3)-- (4)-- (5)-- (1);
			\end{tikzpicture}
			\caption{\(M\)-Tposy}
			\label{fig:Tposy}
		\end{subfigure}
		\hfill
		\begin{subfigure}{0.32\textwidth}
			\centering
			\def \Vescala{1}
			\def \opacidad{0.2}
			\begin{tikzpicture}[thick,scale=0.65,baseline=(0.base)]
				\node[draw,circle, scale=0.6] (0) at (-2,1) {0};
				\node[draw,red,circle,scale=0.6] (1) at (-4,2) {1};
				\node[draw,circle,scale=0.6] (2) at (-5,1) {2};
				\node[draw,circle,scale=0.6] (3) at (-5,-1) {3};
				\node[draw,circle,scale=0.6] (4) at (-4,-2) {4};
				\node[draw,circle,scale=0.6] (5) at (-3,0) {5};
				\node[draw,circle,scale=0.6] (6) at (-1,3) {6};
				\node[draw,circle,scale=0.6] (7) at (0,2) {7};
				\node[draw,circle,scale=0.6] (8) at (0,0) {8};
				\node[draw,red,circle,scale=0.6,fill = cyan!40!white] (9) at (-2,-1) {{$\mathbf{9}$}};
				
				\foreach \from/\to in {1/2,1/5,3/4,0/6,7/8,9/5,9/4,0/5,5/8,2/0} {
					\path[draw] (\from) -- (\to);
				}
				\foreach \from/\to in {2/3,4/5,1/0,6/7,8/9} {
					\path [draw, decorate, decoration={snake, segment length=3mm, amplitude=4}, very thick, red] (\from) -- (\to);
				}
				\draw[cyan, opacity=2*\opacidad, line width=10pt, rounded corners] 
				(9) -- (5) -- (4) -- (9) -- (8) -- (5)-- (4)-- (3)-- (2)-- (0)-- (1)-- (2)-- (3)-- (4)-- (5)-- (1)-- (0)-- (6)-- (7)-- (8)-- (9);
			\end{tikzpicture}
			\caption{\(M\)-Jposy}
			\label{fig:Jposy}
		\end{subfigure}
		
		\caption{Examples of posy configurations}
		\label{fig:allposies}
	\end{figure}
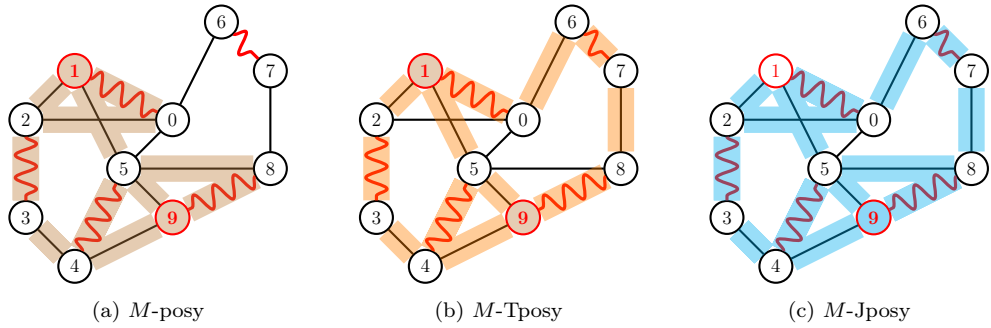

%
%
%
%
%
	
	In \cref{fig:Jposy}, a maximum matching \( M \) of the graph \( G \) is shown in red. With respect to this matching, there are several \( M \)-Jposies in \( G \). The one highlighted in cyan in \cref{fig:Jposy} is formed by the blossom \( 9,5,4,9 \), whose base is vertex \( 9 \) (appearing twice), and the \( M \)-alternating walk:
	\[
	9,8,5,4,3,2,0,1,2,3,4,5,1,0,6,7,8,9.
	\]
	
	Note that any odd cycle appearing in the \( M \)-alternating walk of an \( M \)-Jflower or an \( M \)-Jposy is a blossom with respect to \( M \).
	
	The next results are useful observations about Jposies.
	 
	\begin{lemma}\label[lemma]{triviallemma}
		Let \( G \) be a graph and let \( M \) be a maximum matching of \( G \). If \( H \) is an \( M \)-Jposy of \( G \), then for any \( u, v \in V(H) \), there exists both an \( M \)-\( mm \)-alternating walk and an \( M \)-\( nn \)-alternating walk from \( u \) to \( v \).
	\end{lemma}
	
	\begin{lemma}\label[lemma]{adfw}
		Let \( G \) be a graph and let \( M \) be a maximum matching of \( G \). If \( H \) is an \( M \)-Jposy of \( G \), and \( u, v \in V(H) \) such that there exists an \( M \)-\( nn \)-alternating walk from \( u \) to \( v \), then either there exists an \( M \)-\( nn \)-alternating path from \( u \) to \( v \), or there exists an \( M \)-\( nn \)-alternating walk from \( u \) to \( v \) whose first self-intersection forms an \( M \)-blossom.
	\end{lemma}
	
	\begin{proof}
		If there is no \( M \)-\( nn \)-alternating path from \( u \) to \( v \), consider the \( M \)-\( nn \)-alternating walk from \( u \) to \( v \). Iteratively remove all even cycles from the walk until the first self-intersection forms an odd cycle. This odd cycle is an \( M \)-blossom, as required.
	\end{proof}

	%
	%
	%
	%
	%
	%
	
	\section{Jposies mark the same vertices as posies}
	
	Posies are important because of the following result:
	\begin{theorem}[Sterboul's Characterization of K\H{o}nig-Egerv\'ary graphs \cite{sterboul1979characterization}] \label{theo:Sterboul}
		For a graph \( G \), the following properties are equivalent:
		\begin{enumerate}
			\item \( G \) is not a K\H{o}nig--Egerv\'{a}ry graph;
			\item if \( M\in \mathcal{M}(G) \), then there exists a \(M\)-flower or a \(M\)-posy;
			\item there exists a flower or a posy in \( G \) relative to some maximum matching of \( G \).
		\end{enumerate}
	\end{theorem}

\begin{lemma}\label[lemma]{lem:tech1}
Let \( G \) be a graph and \( M \) a matching of \( G \). Every closed \( M \)-alternating walk in \( G \) contains an odd cycle, which is an \( M \)-blossom of \( G \).
\end{lemma}

\begin{proof}
Let \( W \) be a closed \( M \)-alternating walk in \( G \) starting at a vertex \( w_0 \). Suppose, for contradiction, that \( W \) contains only even cycles. Then the walk cannot return to \( w_0 \) using an edge from \( M \), as each even cycle preserves parity and cannot close the alternating pattern. This contradicts the assumption that \( W \) is a closed \( M \)-alternating walk. Therefore, \( W \) must contain at least one odd cycle. Since it is an \( M \)-alternating walk and the cycle has an odd number of edges with exactly half of them in \( M \), it forms an \( M \)-blossom.
\end{proof}

\begin{lemma}\label[lemma]{lem:noKE1}
	If \( G \) contains a Jposy relative to a maximum matching \( M \), then \( G \) is not a K\H{o}nig–Egerv\'{a}ry graph.
\end{lemma}

\begin{proof}
	Let \( M \) be a maximum matching of \( G \), and suppose that \( G \) contains an \( M \)-Jposy.
	
	First, assume the \( M \)-Jposy contains two distinct \( M \)-blossoms connected by an \( M \)-\(mm\)-alternating walk. Starting at the base of one blossom, traverse the \( M \)-alternating walk, discarding all even cycles encountered along the way. This process must either lead to the base of the second blossom, or result in an odd cycle within the walk.
	
	In either case, this sub-configuration corresponds to an \( M \)-posy in \( G \). By \cref{theo:Sterboul}, \( G \) is not a K\H{o}nig–Egerv\'{a}ry graph.
	
	Now consider the case where the \( M \)-Jposy contains only one \( M \)-blossom. The associated \( M \)-\(mm\)-alternating walk begins and ends at the base of the blossom, forming a closed walk. By \cref{lem:tech1}, any such closed \( M \)-alternating walk must contain an odd cycle, which is itself an \( M \)-blossom. Therefore, we can extract a second \( M \)-blossom from the walk, reducing to the previous case.
	
	In both situations, \( G \) contains an \( M \)-posy. Therefore, by \cref{theo:Sterboul}, \(G\) is not a K\H{o}nig–Egerv\'{a}ry graph.
\end{proof}

	%
	%
	%
	
We now state Theorem 7 from \cite{bonomo2013forbidden}, restated in our notation.

\begin{theorem}[\cite{bonomo2013forbidden}]\label{safe}
	A graph \( G \) is a K\H{o}nig–Egerv\'ary graph if and only if it contains no \( M \)-Tposy and no \( M \)-flower, for any \( M \in \mathcal{M}(G) \).
\end{theorem}

A graph \( G \) is called a \emph{Jposy-graph} if there exists a perfect matching \( M \in \mathcal{M}(G) \) such that \( G \) contains an \( M \)-Jposy that uses all the vertices and all the edges of \( G \). In a similar way we define \emph{posy-graph}, \emph{Tposy-graph}, \emph{Jflower-graph} and \emph{flower-graph}. Note that Tposy are even subdivisions of Barbell graphs or of \(K_{4}\), \cite{bonomo2013forbidden}.

Let $G$ be a graph and let $P=uPv$ be an odd path from \(u\) to \(v\) such that \(V(G)\cap V(P) = \{u,v\}\). Note that \(u\) and \(v\) can be the same vertex of \(G\). With $G+P$ we denote the graph such that $V(G+P) = V(G) \cup V(P)$ and $E(G+P) \cap E(P)\cup E(P) $. We say that \(P\) is and \emph{odd ear} of \(G\).

\begin{lemma}\label[lemma]{Lem_T}
	If $G$ is an Tposy-graph and $P=uPv$ an odd ear of \(G\). Then every vertex of $G+P$ belongs to an Tposy of \(G+P\).
\end{lemma}	
\begin{proof}
	Since Tposy-graphs are even subdivisions of a Barbell or a \(K4\), we break the proof in two. One when \(G\) is an even subdivision of a Barbell graphs, and the other when \(G\) is an even subdivision of a\(K_{4}\).
	
	Let \(G\) be an even subdivision of a Barbell graph. Let \( H \) be the \( M \)-Tposy that defines \( G \), consisting of two \( M \)-blossoms \( B_1 \) and \( B_2 \) with bases \( b_1 \) and \( b_2 \), respectively, and an \( M \)-\( mm \)-alternating path \( R = b_1 R b_2 \) that is internally disjoint from both blossoms. From now on, \(M\) is extended to a perfect matching of \(G+P\). 
	
	If \(u=v\) the result is clear, because the ear form a new \(M\)-blossoms \(B_{3}\) with base \(u\), \(M\)-blossom, together \( M \)-\( mm \)-alternating path \(u,M(u),...,b\), where \(b\) is the base of one of the \(M\)-blossoms \(B_{1}\) or \(B_{2}\) form a Tposy of \(G+P\) relative to \(M\). 
	
	Assume that \(u\neq v\). We consider four main cases depending on how the ear \( P \) connects vertices in \( G \): (1) both endpoints are on the path \( R \), (2) both are in the same blossom, (3) one is in a blossom and the other on the path, and (4) each is in a different blossom. Each case has two or three subcases, depending on whether the \( M \)-alternating path between \( u \) and \( v \) in \( G \) is of type \( mm \), \( nn \), or \( nm \). 
	
	Our goal is to prove that the internal vertices of \( P \) belong to a Tposy of \( G + P \).
	
	\textbf{Case 1:} Both \( u \) and \( v \) lie on the path \( R \) (see \cref{fig:odioapane1}).
	
	\textbf{Case 1.1:} The path \( u Q v \) in \( G \) is an \( M \)-\( mm \)-alternating subpath. This corresponds to the ear \( 2P3 \) in \cref{fig:odioapane1}. Let \( M^* \) be the symmetric difference of \( M \) with the cycle \( u P v Q u \). Then \( M^* \) is a maximum matching. Note that \( B_1 \) and \( B_2 \) remain \( M^* \)-blossoms, and the path \( b_1 R u P v Q u R b_2 \) is \( M^* \)-\( mm \)-alternating. Together, they form an \( M^* \)-Tposy that includes all vertices of \( P \).
	
	\textbf{Case 1.2:} The path \( u Q v \) in \( G \) is an \( M \)-\( nn \)-alternating subpath. This corresponds to the ear \( 3P4 \) in \cref{fig:odioapane1}. The path \( R^* = b_1 R u P v Q R b_2 \) is an \( M \)-\( mm \)-alternating path. The blossoms \( B_1 \) and \( B_2 \), along with \( R^* \), form an \( M \)-Tposy that includes all vertices of \( P \).
	
	\textbf{Case 1.3:} The path \( u Q v \) in \( G \) is an \( M \)-\( nm \)-alternating subpath. This corresponds to the ear \( 3P5 \) in \cref{fig:odioapane1}. The cycle \( u P v R u \) is an \( M \)-blossom \( B_3 \) in \( G + P \) containing all vertices of \( P \). Then, either \( B_1 \) or \( B_2 \), together with \( B_3 \) and one of the paths \( b_1 R u \) or \( b_2 R u \), form an \( M \)-Tposy that includes all vertices of \( P \). In \cref{fig:odioapane1}, for example, the blossom \( 2,0,1 \), the path \( 2,3 \), and the blossom \( 3,3,4,5,13,12 \) illustrate this configuration.

	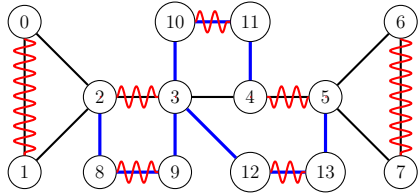
\begin{figure}[H]
		\centering	
		\begin{tikzpicture}[scale=2]
			
			\node[draw,circle,scale=0.6] (0) at (1.5,6.5) {\(0\)};
			\node[draw,circle,scale=0.6] (1) at (1.5,5.5) {\(1\)};
			\node[draw,circle,scale=0.6] (2) at (2.0,6.0) {\(2\)};
			\node[draw,circle,scale=0.6] (3) at (2.5,6.0) {\(3\)};
			\node[draw,circle,scale=0.6] (4) at (3.0,6.0) {\(4\)};
			\node[draw,circle,scale=0.6] (5) at (3.5,6.0) {\(5\)};
			\node[draw,circle,scale=0.6] (6) at (4.0,6.5) {\(6\)};
			\node[draw,circle,scale=0.6] (7) at (4.0,5.5) {\(7\)};
			\node[draw,circle,scale=0.6] (8) at (2.0,5.5) {\(8\)};
			\node[draw,circle,scale=0.6] (9) at (2.5,5.5) {\(9\)};
			\node[draw,circle,scale=0.6] (10) at (2.5,6.5) {\(10\)};
			\node[draw,circle,scale=0.6] (11) at (3.0,6.5) {\(11\)};
			\node[draw,circle,scale=0.6] (12) at (3.0,5.5) {\(12\)};
			\node[draw,circle,scale=0.6] (13) at (3.5,5.5) {\(13\)};
			
			\foreach \from/\to in { {0/1}, {1/2}, {2/0}, {2/3}, {3/4}, {4/5}, {5/6}, {6/7}, {7/5}} {
				\path[draw, thick] (\from) -- (\to);}
			
			\foreach \from/\to in { {2/8}, {8/9}, {9/3}, {3/10}, {10/11}, {11/4}, {3/12}, {12/13}, {13/5}} {
				\path[draw, blue, very thick] (\from) -- (\to);}	
			
			\foreach \from/\to in { {0/1}, {2/2}, {3/3}, {2/3}, {4/4}, {5/5}, {5/4}, {6/7}, {8/9}, {10/11}, {12/13}} {
				\path[decorate, decoration={snake,amplitude=4, segment length=2mm},draw=red,thick] (\from) -- (\to);}
		\end{tikzpicture}
		\caption{Case 1: Three ears: \(2P3=2,8,9,3\), \(3P4=3,10,11,4\) and \(3P5=3,12,13,5\).}\label{fig:odioapane1}
	\end{figure}

%
%
%
%
%
%
	
	%
	%
	%
	
	\textbf{Case 2:} Both \( u \) and \( v \) are vertices of one of the blossoms (see \cref{fig:odioapane2}). Without loss of generality, assume that \( u, v \in V(B_1) \).
	
	\textbf{Case 2.1:} The path \( u Q v \) in \( G \) is an \( M \)-\( mm \)-alternating subpath. This corresponds to the ear \( 9P10 \) in \cref{fig:odioapane1}. Let \( M^* \) be the symmetric difference of \( M \) with the cycle \( u P v Q u \). Then \( M^* \) is a maximum matching. The blossom \( B_2 \) remains an \( M^* \)-blossom, and the path \( R \) is \( M^* \)-\( mm \)-alternating. The cycle formed by \( b_1 B_1 u P v B_1 b_1 \) (or symmetrically \( b_1 B_1 v P u B_1 b_1 \)) is an \( M^* \)-blossom of \( G + P \), denoted \( B_3 \), which contains all vertices of the ear \( P \). Thus, \( B_2 \), \( B_3 \), and the path \( R \) form an \( M^* \)-Tposy that includes all vertices of \( P \).
	
	\textbf{Case 2.2:} The path \( u Q v \) in \( G \) is an \( M \)-\( nn \)-alternating subpath. This corresponds to the ear \( 2P3 \) in \cref{fig:odioapane1}. In this case, the cycle \( b_1 B_1 u P v B_1 b_1 \) (or \( b_1 B_1 v P u B_1 b_1 \)) is an \( M \)-blossom of \( G + P \), denoted \( B_3 \), which contains all the vertices of \( P \). Then \( B_3 \), \( B_2 \), and the path \( R \) form an \( M \)-Tposy containing all vertices of \( P \).
	
	\textbf{Case 2.3:} The path \( u Q v \) in \( G \) is an \( M \)-\( nm \)-alternating subpath. This corresponds to the ear \( 6P8 \) in \cref{fig:odioapane1}. The cycle \( b_1 B_1 u P v B_2 u \) forms an \( M \)-blossom in \( G + P \), denoted \( B_3 \), that contains all vertices of the ear \( P \). Then \( B_3 \), \( B_2 \), and the path \( u B_1 b_1 R \) form an \( M \)-Tposy that includes all vertices of \( P \).

	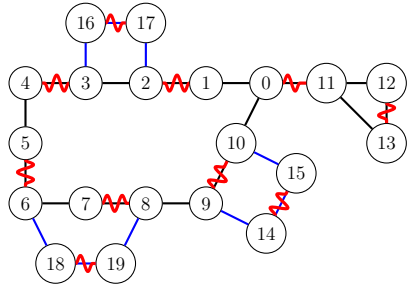
\begin{figure}[h]
		\centering
		\begin{tikzpicture}[scale=0.8]
			
			\node[draw,circle,scale=0.6] (0) at (5.5,6.0) {\( 0 \)};
			\node[draw,circle,scale=0.6] (1) at (4.5,6.0) {\( 1 \)};
			\node[draw,circle,scale=0.6] (2) at (3.5,6.0) {\( 2 \)};
			\node[draw,circle,scale=0.6] (3) at (2.5,6.0) {\( 3 \)};
			\node[draw,circle,scale=0.6] (4) at (1.5,6.0) {\( 4 \)};
			\node[draw,circle,scale=0.6] (5) at (1.5,5.0) {\( 5 \)};
			\node[draw,circle,scale=0.6] (6) at (1.5,4.0) {\( 6 \)};
			\node[draw,circle,scale=0.6] (7) at (2.5,4.0) {\( 7 \)};
			\node[draw,circle,scale=0.6] (8) at (3.5,4.0) {\( 8 \)};
			\node[draw,circle,scale=0.6] (9) at (4.5,4.0) {\( 9 \)};
			\node[draw,circle,scale=0.6] (10) at (5.0,5.0) {\( 10\)};
			\node[draw,circle,scale=0.6] (11) at (6.5,6.0) {\( 11\)};
			\node[draw,circle,scale=0.6] (12) at (7.5,6.0) {\( 12\)};
			\node[draw,circle,scale=0.6] (13) at (7.5,5.0) {\( 13\)};
			\node[draw,circle,scale=0.6] (14) at (5.5,3.5) {\( 14\)};
			\node[draw,circle,scale=0.6] (15) at (6.0,4.5) {\( 15\)};
			\node[draw,circle,scale=0.6] (16) at (2.5,7.0) {\( 16\)};
			\node[draw,circle,scale=0.6] (17) at (3.5,7.0) {\( 17\)};
			\node[draw,circle,scale=0.6] (18) at (2.0,3.0) {\( 18\)};
			\node[draw,circle,scale=0.6] (19) at (3.0,3.0) {\( 19\)};
			
			\foreach \from/\to in { {0/1}, {1/2}, {2/3}, {3/4}, {4/5}, {5/6}, {6/7}, {7/8}, {8/9}, {10/9}, {10/0}, {0/11}, {11/12}, {12/13}, {13/11}} {
				\path[draw, thick ] (\from) -- (\to);}
			
			\foreach \from/\to in { {14/9}, {10/15}, {15/14}, {3/16}, {16/17}, {17/2}, {18/19}, {6/18}, {19/8}} {
				\path[draw,blue, thick] (\from) -- (\to);}
			
			\foreach \from/\to in { {0/11}, {12/13}, {1/2}, {3/4}, {5/6}, {7/8}, {9/10}, {18/19}, {14/15}, {16/17}} {
				\path[decorate, decoration={snake,amplitude=3, segment length=2mm},draw=red,very thick] (\from) -- (\to);}
		\end{tikzpicture}
		\caption{Case 2: Three ears: \(2P3=2,17,16,3\), \(9P10=9,14,15,10\) and \(6P8=6,18,19,8\).}\label{fig:odioapane2}
	\end{figure}

	\textbf{Case 3:} The vertex \( u \) belongs to the blossom \( B_1 \), and \( v \) belongs to the path \( R \) of the \( M \)-Tposy \( H \); see \cref{fig:odioapane3}. The ears \( 2P9 \), \( 2P10 \), and \( 8P10 \) all correspond to this general configuration.
	
	\textbf{Case 3.1:} The path \( uQv \) in \( G \) is an \( M \)-\( mm \)-alternating subpath. This corresponds to the ear \( 2P9 \) in \cref{fig:odioapane3}, where \( Q = 2,1,0,9 \). Let \( M^* = M \oplus C \), where \( C = uQvPu \) is the cycle formed by joining the ear \( P \) and the path \( Q \). Then \( M^* \) is a maximum matching. Under \( M^* \), both \( B_1 \) and \( B_2 \) remain \( M^* \)-blossoms, with \( u \) now serving as the base of \( B_1 \). The path \( uPvRb_2 \) is an \( M^* \)-\( mm \)-alternating path that is internally disjoint from both blossoms. Together, these form an \( M^* \)-Tposy containing all the vertices of \( P \).
	
	\textbf{Case 3.2:} The path \( uQv \) in \( G \) is an \( M \)-\( mn \)-alternating subpath. This corresponds to the ear \( 2P10 \) in \cref{fig:odioapane3}, where \( Q = 2,1,0,9,10 \). Let \( B_3 \) be the \( M \)-blossom formed by the closed \( M \)-alternating walk \( vPuQv \). Then \( B_3 \), together with the \( M \)-\( mm \)-alternating path \( vRb_2 \), and the \( M \)-blossom \( B_2 \), forms an \( M \)-Tposy that includes all the vertices of \( P \).

	\begin{figure}[h]
		\centering
		\begin{tikzpicture}[scale=1.6]
			\node[draw,circle,scale=0.6] (0) at (3.5,5.5) {\(0\)};
			\node[draw,circle,scale=0.6] (1) at (3.0,6.0) {\(1\)};
			\node[draw,circle,scale=0.6] (2) at (2.5,6.0) {\(2\)};
			\node[draw,circle,scale=0.6] (3) at (2.0,6.0) {\(3\)};
			\node[draw,circle,scale=0.6] (4) at (1.5,6.0) {\(4\)};
			\node[draw,circle,scale=0.6] (5) at (1.5,5.0) {\(5\)};
			\node[draw,circle,scale=0.6] (6) at (2.0,5.0) {\(6\)};
			\node[draw,circle,scale=0.6] (7) at (2.5,5.0) {\(7\)};
			\node[draw,circle,scale=0.6] (8) at (3.0,5.0) {\(8\)};
			\node[draw,circle,scale=0.6] (9) at (4.0,5.5) {\(9\)};
			\node[draw,circle,scale=0.6] (10) at (4.5,5.5) {\(10\)};
			\node[draw,circle,scale=0.6] (11) at (5.0,5.5) {\(11\)};
			\node[draw,circle,scale=0.6] (12) at (5.5,6.0) {\(12\)};
			\node[draw,circle,scale=0.6] (13) at (5.5,5.0) {\(13\)};
			\node[draw,circle,scale=0.6] (14) at (3.5,4.5) {\(14\)};
			\node[draw,circle,scale=0.6] (15) at (4.0,4.5) {\(15\)};
			\node[draw,circle,scale=0.6] (16) at (3.0,6.5) {\(16\)};
			\node[draw,circle,scale=0.6] (17) at (3.5,6.5) {\(17\)};
			\node[draw,circle,scale=0.6] (18) at (3.0,7.0) {\(18\)};
			\node[draw,circle,scale=0.6] (19) at (3.5,7.0) {\(19\)};
			\node[draw,circle,scale=0.6] (20) at (3.5,4.0) {\(20\)};
			\node[draw,circle,scale=0.6] (21) at (4.0,4.0) {\(21\)};
			
			\foreach \from/\to in { {0/1}, {1/2}, {2/3}, {3/4}, {4/5}, {5/6}, {6/7}, {7/8}, {8/0}, {0/9}, {10/9}, {10/11}, {11/12}, {12/13}, {13/11}} {
				\path[draw, thick] (\from) -- (\to);}
			
			\foreach \from/\to in { {8/14}, {14/15}, {10/15}, {16/16}, {2/2}, {2/16}, {16/17}, {17/9}, {2/18}, {18/19}, {19/10}, {8/20}, {20/21}, {21/21}, {11/21}} {
				\path[draw, thick, blue] (\from) -- (\to);}
			
			\foreach \from/\to in { {0/9}, {10/11}, {12/13}, {1/2}, {3/4}, {6/5}, {7/8}, {16/17}, {18/19}, {14/15}, {20/21}} {
				\path[decorate, decoration={snake,amplitude=3, segment length=2mm},draw=red, very thick] (\from) -- (\to);}
		\end{tikzpicture}
		\caption{Case 3: Four ears: \(2P9=2,16,17,9\), \(8P10=8,14,15,10\),  \(2P10=2,18,19,10\), and \(8P11=8,20,21,11\).}\label{fig:odioapane3}
	\end{figure}
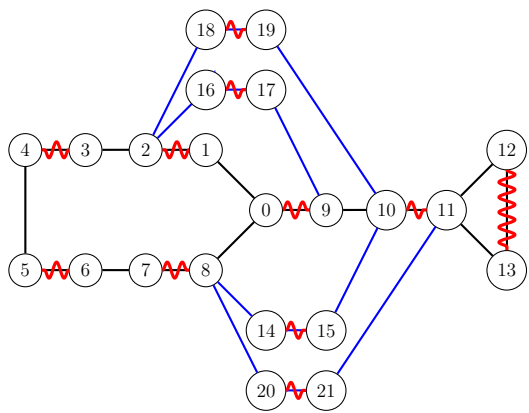

	\textbf{Case 4:} The endpoints \( u \) and \( v \) of the ear \( P \) lie in different blossoms of the \( M \)-Tposy \( H \); specifically, \( u \in V(B_1) \) and \( v \in V(B_2) \), see \cref{fig:odioapane4}. This applies to the ears \(2P7\), \( 2P6 \), \( 4P9 \), and \( 4P8 \), which are structurally similar.
	
	Let \( Q \) be the path in \( G \) from \( u \) to \( v \), and assume that \( uQv \) is an \( M \)-alternating path, see \cref{fig:odioapane4} where \(uPv=2,11,13,7\) and \(uQv=2,(M(2)=1),0,5,6,(M(6)=7)\). Define \( M^* = M \triangle C \), where \( C = u Q v P u \) is the cycle formed by joining \( P \) and \( Q \). Then \( M^* \) is a maximum matching. 
	
	Under \( M^* \), \( B_1 \) and \( B_2 \) are the blossoms, but now with \( u \) and \( v \) as their respective bases. The path \( P \) is an \( M^* \)-\( mm \)-alternating path that is internally disjoint from both \( M^* \)-blossoms. Hence, the configuration consisting of \( B_1 \), \( B_2 \), and \( P \) forms an \( M^* \)-Tposy that includes all vertices of \( P \). \qedhere
	
	\begin{figure}[H]
		\centering
		\begin{tikzpicture}[scale=0.9] 
			
			\node[draw,circle,scale=0.6] (0) at (3.0,6.0) {\(0\)};
			\node[draw,circle,scale=0.6] (1) at (2.5,7.0) {\(1\)};
			\node[draw,circle,scale=0.6] (2) at (1.5,7.0) {\(2\)};
			\node[draw,circle,scale=0.6] (3) at (1.5,5.0) {\(3\)};
			\node[draw,circle,scale=0.6] (4) at (2.5,5.0) {\(4\)};
			\node[draw,circle,scale=0.6] (5) at (4.0,6.0) {\(5\)};
			\node[draw,circle,scale=0.6] (6) at (4.5,7.0) {\(6\)};
			\node[draw,circle,scale=0.6] (7) at (5.5,7.0) {\(7\)};
			\node[draw,circle,scale=0.6] (8) at (5.5,5.0) {\(8\)};
			\node[draw,circle,scale=0.6] (9) at (4.5,5.0) {\(9\)};
			\node[draw,circle,scale=0.6] (10) at (3.0,8.0) {\(10\)};
			\node[draw,circle,scale=0.6] (11) at (3.0,9.0) {\(11\)};
			\node[draw,circle,scale=0.6] (12) at (4.0,8.0) {\(12\)};
			\node[draw,circle,scale=0.6] (13) at (4.0,9.0) {\(13\)};
			\node[draw,circle,scale=0.6] (14) at (3.0,4.0) {\(14\)};
			\node[draw,circle,scale=0.6] (15) at (4.0,4.0) {\(15\)};
			\node[draw,circle,scale=0.6] (16) at (3.0,3.0) {\(16\)};
			\node[draw,circle,scale=0.6] (17) at (4.0,3.0) {\(17\)};
			
			\foreach \from/\to in { {0/1}, {1/2}, {2/3}, {3/4}, {4/0}, {0/5}, {5/6}, {6/7}, {7/8}, {8/9}} {
				\path[draw, very thick] (\from) -- (\to);}
			
			\foreach \from/\to in { {9/5}, {2/10}, {2/11}, {10/12}, {11/13}, {12/6}, {13/7}, {14/15}, {16/17}, {4/14}, {15/9}, {4/16}, {17/8}} {
				\path[draw, blue ,very thick] (\from) -- (\to);}
			
			\foreach \from/\to in { {16/17}, {0/5}, {10/12}, {11/13}, {14/15}, {2/1}, {6/7}, {9/8}, {3/4}} {
				\path[decorate, decoration={snake,amplitude=3, segment length=2mm},draw=red, very thick] (\from) -- (\to);}
		\end{tikzpicture}
		\caption{Case 3: Four ears: \(2P7=2,11,13,7\), \(2P6=2,10,12,6\),  \(4P9=4,14,15,9\), and \(4P8=4,16,17,8\).}\label{fig:odioapane4}
	\end{figure}

	Let \(G\) be an even subdivision of a \(K_{4}\). An analysis similar to the one that we do for the case when \(G\) is an even subdivision of a Barbell graph allow us to prove that all the internal vertices of \(P\) belong to a Tposy of \(G+P\). If you follow that path you face with 4 major cases (exploiting symmetries), each of them with sub cases. But it is possible another approach. 
	
	Let $G$ be an even subdivision of $K_{4}$. The graph $G$ can be constructed
	from four vertices $v_{1},v_{2},v_{3}$ and $v_{4}$, where for each pair of distinct	vertices $v_{i},v_{j}$ with $i\neq j$ an ear $P_{ij}$ of odd length
	is attached. We now add to $G$ an additional ear $P=u_{1}\dotsm u_{k}$
	with $u_{1}\neq u_{k}$. We use that \(G\) has exactly three perfect matchings.
	
	$ $\\
	\textbf{Case 1.} $u_{1},u_{k}\in V(v_{1}Pv_{2})$. This case is the same as Case 1 when \(G\) is an even subdivision of a Barbell.
	
	$ $\\
	\textbf{Case 2.} $u_{1}\in V\left(v_{1}Pv_{2}\right)$ and $u_{k}\in V\left(v_{1}Pv_{3}\right)$. 
	
	$ $
	
	\textbf{Case 2.1.} $\text{dist}(v_{1},u_{1})\equiv 0 \pmod{2}$ and $\text{dist}(v_{1},u_{k})\equiv 0 \pmod{2}$. See \cref{K4FiguraCase21}. Choosing any perfect matching of \(G\) such that the paths from \(v_1\) to \(v_{2}\), and from \(v_{1}\) to \(v_{3}\) are \(nn\)-alternating.  We obtain a Tposy that uses \(u_{1}Pu_{k}\) as part of one of its blossoms.
	
	\begin{figure}[H]
		
		\begin{center}			
			
			\tikzset{every picture/.style={line width=0.75pt}} 
			
			\hspace{-3cm}
			\begin{tikzpicture}[x=0.75pt,y=0.75pt,yscale=-1,xscale=1, scale=.85]
				
				\draw    (438.47,243.39) -- (193.37,243.39) ;
				\draw [shift={(193.37,243.39)}, rotate = 180] [color={rgb, 255:red, 0; green, 0; blue, 0 }  ][fill={rgb, 255:red, 0; green, 0; blue, 0 }  ][line width=0.75]      (0, 0) circle [x radius= 3.35, y radius= 3.35]   ;
				\draw [shift={(438.47,243.39)}, rotate = 180] [color={rgb, 255:red, 0; green, 0; blue, 0 }  ][fill={rgb, 255:red, 0; green, 0; blue, 0 }  ][line width=0.75]      (0, 0) circle [x radius= 3.35, y radius= 3.35]   ;
				\draw    (438.47,243.39) -- (315.92,31.13) ;
				\draw [shift={(315.92,31.13)}, rotate = 240] [color={rgb, 255:red, 0; green, 0; blue, 0 }  ][fill={rgb, 255:red, 0; green, 0; blue, 0 }  ][line width=0.75]      (0, 0) circle [x radius= 3.35, y radius= 3.35]   ;
				\draw [shift={(438.47,243.39)}, rotate = 240] [color={rgb, 255:red, 0; green, 0; blue, 0 }  ][fill={rgb, 255:red, 0; green, 0; blue, 0 }  ][line width=0.75]      (0, 0) circle [x radius= 3.35, y radius= 3.35]   ;
				\draw    (193.37,243.39) -- (315.92,31.13) ;
				\draw [shift={(315.92,31.13)}, rotate = 300] [color={rgb, 255:red, 0; green, 0; blue, 0 }  ][fill={rgb, 255:red, 0; green, 0; blue, 0 }  ][line width=0.75]      (0, 0) circle [x radius= 3.35, y radius= 3.35]   ;
				\draw [shift={(193.37,243.39)}, rotate = 300] [color={rgb, 255:red, 0; green, 0; blue, 0 }  ][fill={rgb, 255:red, 0; green, 0; blue, 0 }  ][line width=0.75]      (0, 0) circle [x radius= 3.35, y radius= 3.35]   ;
				\draw    (438.47,243.39) -- (315.92,150.07) ;
				\draw [shift={(315.92,150.07)}, rotate = 217.29] [color={rgb, 255:red, 0; green, 0; blue, 0 }  ][fill={rgb, 255:red, 0; green, 0; blue, 0 }  ][line width=0.75]      (0, 0) circle [x radius= 3.35, y radius= 3.35]   ;
				\draw [shift={(438.47,243.39)}, rotate = 217.29] [color={rgb, 255:red, 0; green, 0; blue, 0 }  ][fill={rgb, 255:red, 0; green, 0; blue, 0 }  ][line width=0.75]      (0, 0) circle [x radius= 3.35, y radius= 3.35]   ;
				\draw    (193.37,243.39) -- (315.92,150.07) ;
				\draw [shift={(315.92,150.07)}, rotate = 322.71] [color={rgb, 255:red, 0; green, 0; blue, 0 }  ][fill={rgb, 255:red, 0; green, 0; blue, 0 }  ][line width=0.75]      (0, 0) circle [x radius= 3.35, y radius= 3.35]   ;
				\draw [shift={(193.37,243.39)}, rotate = 322.71] [color={rgb, 255:red, 0; green, 0; blue, 0 }  ][fill={rgb, 255:red, 0; green, 0; blue, 0 }  ][line width=0.75]      (0, 0) circle [x radius= 3.35, y radius= 3.35]   ;
				\draw    (315.92,150.07) -- (315.92,31.13) ;
				\draw [shift={(315.92,31.13)}, rotate = 270] [color={rgb, 255:red, 0; green, 0; blue, 0 }  ][fill={rgb, 255:red, 0; green, 0; blue, 0 }  ][line width=0.75]      (0, 0) circle [x radius= 3.35, y radius= 3.35]   ;
				\draw [shift={(315.92,150.07)}, rotate = 270] [color={rgb, 255:red, 0; green, 0; blue, 0 }  ][fill={rgb, 255:red, 0; green, 0; blue, 0 }  ][line width=0.75]      (0, 0) circle [x radius= 3.35, y radius= 3.35]   ;
				\draw    (282.47,243.39) -- (193.37,243.39) ;
				\draw [shift={(193.37,243.39)}, rotate = 180] [color={rgb, 255:red, 0; green, 0; blue, 0 }  ][fill={rgb, 255:red, 0; green, 0; blue, 0 }  ][line width=0.75]      (0, 0) circle [x radius= 3.35, y radius= 3.35]   ;
				\draw [shift={(282.47,243.39)}, rotate = 180] [color={rgb, 255:red, 0; green, 0; blue, 0 }  ][fill={rgb, 255:red, 0; green, 0; blue, 0 }  ][line width=0.75]      (0, 0) circle [x radius= 3.35, y radius= 3.35]   ;
				\draw    (438.47,243.39) -- (349.37,243.39) ;
				\draw [shift={(349.37,243.39)}, rotate = 180] [color={rgb, 255:red, 0; green, 0; blue, 0 }  ][fill={rgb, 255:red, 0; green, 0; blue, 0 }  ][line width=0.75]      (0, 0) circle [x radius= 3.35, y radius= 3.35]   ;
				\draw [shift={(438.47,243.39)}, rotate = 180] [color={rgb, 255:red, 0; green, 0; blue, 0 }  ][fill={rgb, 255:red, 0; green, 0; blue, 0 }  ][line width=0.75]      (0, 0) circle [x radius= 3.35, y radius= 3.35]   ;
				\draw    (238.47,164.07) -- (193.37,243.39) ;
				\draw [shift={(193.37,243.39)}, rotate = 119.62] [color={rgb, 255:red, 0; green, 0; blue, 0 }  ][fill={rgb, 255:red, 0; green, 0; blue, 0 }  ][line width=0.75]      (0, 0) circle [x radius= 3.35, y radius= 3.35]   ;
				\draw [shift={(238.47,164.07)}, rotate = 119.62] [color={rgb, 255:red, 0; green, 0; blue, 0 }  ][fill={rgb, 255:red, 0; green, 0; blue, 0 }  ][line width=0.75]      (0, 0) circle [x radius= 3.35, y radius= 3.35]   ;
				\draw    (315.92,31.13) -- (270.82,110.45) ;
				\draw [shift={(270.82,110.45)}, rotate = 119.62] [color={rgb, 255:red, 0; green, 0; blue, 0 }  ][fill={rgb, 255:red, 0; green, 0; blue, 0 }  ][line width=0.75]      (0, 0) circle [x radius= 3.35, y radius= 3.35]   ;
				\draw [shift={(315.92,31.13)}, rotate = 119.62] [color={rgb, 255:red, 0; green, 0; blue, 0 }  ][fill={rgb, 255:red, 0; green, 0; blue, 0 }  ][line width=0.75]      (0, 0) circle [x radius= 3.35, y radius= 3.35]   ;
				\draw    (349.37,243.39) .. controls (149.76,426.9) and (-14.24,167.9) .. (270.82,110.45) ;
				\draw [color={rgb, 255:red, 255; green, 0; blue, 0 }  ,draw opacity=1 ][line width=1.5]    (238.47,164.07) -- (270.82,110.45) ;
				\draw [color={rgb, 255:red, 255; green, 0; blue, 0 }  ,draw opacity=1 ][line width=1.5]    (349.37,243.39) -- (282.47,243.39) ;
				\draw [color={rgb, 255:red, 255; green, 0; blue, 0 }  ,draw opacity=1 ][line width=1.5]    (315.92,150.07) -- (193.37,243.39) ;
				\draw [color={rgb, 255:red, 255; green, 0; blue, 0 }  ,draw opacity=1 ][line width=1.5]    (438.47,243.39) -- (315.92,31.13) ;
				\draw [color={rgb, 255:red, 126; green, 211; blue, 33 }  ,draw opacity=0.35 ][line width=6]    (193.37,243.39) .. controls (259.47,194.87) and (281.47,176.87) .. (315.92,150.07) ;
				\draw    (438.47,243.39) ;
				\draw [shift={(438.47,243.39)}, rotate = 0] [color={rgb, 255:red, 0; green, 0; blue, 0 }  ][fill={rgb, 255:red, 0; green, 0; blue, 0 }  ][line width=0.75]      (0, 0) circle [x radius= 3.35, y radius= 3.35]   ;
				\draw [shift={(438.47,243.39)}, rotate = 0] [color={rgb, 255:red, 0; green, 0; blue, 0 }  ][fill={rgb, 255:red, 0; green, 0; blue, 0 }  ][line width=0.75]      (0, 0) circle [x radius= 3.35, y radius= 3.35]   ;
				\draw    (270.82,110.45) ;
				\draw [shift={(270.82,110.45)}, rotate = 0] [color={rgb, 255:red, 0; green, 0; blue, 0 }  ][fill={rgb, 255:red, 0; green, 0; blue, 0 }  ][line width=0.75]      (0, 0) circle [x radius= 3.35, y radius= 3.35]   ;
				\draw [shift={(270.82,110.45)}, rotate = 0] [color={rgb, 255:red, 0; green, 0; blue, 0 }  ][fill={rgb, 255:red, 0; green, 0; blue, 0 }  ][line width=0.75]      (0, 0) circle [x radius= 3.35, y radius= 3.35]   ;
				\draw    (315.92,150.07) ;
				\draw [shift={(315.92,150.07)}, rotate = 0] [color={rgb, 255:red, 0; green, 0; blue, 0 }  ][fill={rgb, 255:red, 0; green, 0; blue, 0 }  ][line width=0.75]      (0, 0) circle [x radius= 3.35, y radius= 3.35]   ;
				\draw [shift={(315.92,150.07)}, rotate = 0] [color={rgb, 255:red, 0; green, 0; blue, 0 }  ][fill={rgb, 255:red, 0; green, 0; blue, 0 }  ][line width=0.75]      (0, 0) circle [x radius= 3.35, y radius= 3.35]   ;
				\draw    (238.47,164.07) ;
				\draw [shift={(238.47,164.07)}, rotate = 0] [color={rgb, 255:red, 0; green, 0; blue, 0 }  ][fill={rgb, 255:red, 0; green, 0; blue, 0 }  ][line width=0.75]      (0, 0) circle [x radius= 3.35, y radius= 3.35]   ;
				\draw [shift={(238.47,164.07)}, rotate = 0] [color={rgb, 255:red, 0; green, 0; blue, 0 }  ][fill={rgb, 255:red, 0; green, 0; blue, 0 }  ][line width=0.75]      (0, 0) circle [x radius= 3.35, y radius= 3.35]   ;
				\draw    (193.37,243.39) ;
				\draw [shift={(193.37,243.39)}, rotate = 0] [color={rgb, 255:red, 0; green, 0; blue, 0 }  ][fill={rgb, 255:red, 0; green, 0; blue, 0 }  ][line width=0.75]      (0, 0) circle [x radius= 3.35, y radius= 3.35]   ;
				\draw [shift={(193.37,243.39)}, rotate = 0] [color={rgb, 255:red, 0; green, 0; blue, 0 }  ][fill={rgb, 255:red, 0; green, 0; blue, 0 }  ][line width=0.75]      (0, 0) circle [x radius= 3.35, y radius= 3.35]   ;
				\draw    (282.47,243.39) ;
				\draw [shift={(282.47,243.39)}, rotate = 0] [color={rgb, 255:red, 0; green, 0; blue, 0 }  ][fill={rgb, 255:red, 0; green, 0; blue, 0 }  ][line width=0.75]      (0, 0) circle [x radius= 3.35, y radius= 3.35]   ;
				\draw [shift={(282.47,243.39)}, rotate = 0] [color={rgb, 255:red, 0; green, 0; blue, 0 }  ][fill={rgb, 255:red, 0; green, 0; blue, 0 }  ][line width=0.75]      (0, 0) circle [x radius= 3.35, y radius= 3.35]   ;
				\draw    (349.37,243.39) ;
				\draw [shift={(349.37,243.39)}, rotate = 0] [color={rgb, 255:red, 0; green, 0; blue, 0 }  ][fill={rgb, 255:red, 0; green, 0; blue, 0 }  ][line width=0.75]      (0, 0) circle [x radius= 3.35, y radius= 3.35]   ;
				\draw [shift={(349.37,243.39)}, rotate = 0] [color={rgb, 255:red, 0; green, 0; blue, 0 }  ][fill={rgb, 255:red, 0; green, 0; blue, 0 }  ][line width=0.75]      (0, 0) circle [x radius= 3.35, y radius= 3.35]   ;
				\draw    (315.92,31.13) ;
				\draw [shift={(315.92,31.13)}, rotate = 0] [color={rgb, 255:red, 0; green, 0; blue, 0 }  ][fill={rgb, 255:red, 0; green, 0; blue, 0 }  ][line width=0.75]      (0, 0) circle [x radius= 3.35, y radius= 3.35]   ;
				\draw [shift={(315.92,31.13)}, rotate = 0] [color={rgb, 255:red, 0; green, 0; blue, 0 }  ][fill={rgb, 255:red, 0; green, 0; blue, 0 }  ][line width=0.75]      (0, 0) circle [x radius= 3.35, y radius= 3.35]   ;
				\draw  [color={rgb, 255:red, 80; green, 227; blue, 194 }  ,draw opacity=0.27 ][line width=6]  (149.76,165.3) .. controls (119.76,197.3) and (107.76,239.48) .. (127.76,267.89) .. controls (147.76,296.3) and (196.76,312.3) .. (240.76,306.3) .. controls (284.76,300.3) and (352.76,246.3) .. (349.37,243.39) .. controls (345.98,240.48) and (303.76,241.3) .. (282.47,243.39) .. controls (261.17,245.48) and (206.21,242.56) .. (193.37,243.39) .. controls (180.53,244.22) and (252.76,146.3) .. (238.47,164.07) .. controls (224.17,181.83) and (278.76,108.89) .. (270.82,110.45) .. controls (262.88,112.01) and (179.76,133.3) .. (149.76,165.3) -- cycle ;
				\draw  [color={rgb, 255:red, 189; green, 16; blue, 224 }  ,draw opacity=0.23 ][line width=6]  (438.47,243.39) .. controls (444.33,240.46) and (406.78,189.14) .. (377.19,137.26) .. controls (347.6,85.38) and (319.08,24.37) .. (315.92,31.13) .. controls (312.76,37.89) and (316.76,65.89) .. (316.76,96.89) .. controls (316.76,127.89) and (315.42,144.87) .. (315.92,150.07) .. controls (316.41,155.27) and (344.76,172.89) .. (369.76,190.89) .. controls (394.76,208.89) and (432.6,246.32) .. (438.47,243.39) -- cycle ;
				
				\draw (169,237.4) node [anchor=north west][inner sep=0.75pt]    {$v_{1}$};
				
				\draw (260,90) node [anchor=north][inner sep=0.75pt]    {$u_{1}$};
				\draw (360,270) node [anchor=south][inner sep=0.75pt]    {$u_{k}$};
				\draw (130,210) node [anchor= west][inner sep=0.75pt]    {$u_{1}Pu_{k}$};
				\draw (308,8.4) node [anchor=north west][inner sep=0.75pt]    {$v_{2}$};
				\draw (443,240.4) node [anchor=north west][inner sep=0.75pt]    {$v_{3}$};
				\draw (321,133.4) node [anchor=north west][inner sep=0.75pt]    {$v_{4}$};

			\end{tikzpicture}
			
			\vspace{-3cm}
		\end{center}
		\caption{Even subdivision of \(K_{4}\). Case 2.1}
		\label{K4FiguraCase21}
		
	\end{figure}

	 \textbf{Case 2.2.} $\text{dist}(v_{1},u_{1})\equiv 1  \pmod{2}$ y $\text{dist}(v_{1},u_{k})\equiv 1  \pmod{2}$. See \cref{K4FiguraCase22}. Choosing a perfect matching of \(G\) such that the path from \(v_1\) to \(v_{2}\) is \(nn\)-alternating, and the path from \(v_{1}\) to \(v_{3}\) is \(mm\)-alternating.  We obtain a Tposy that uses \(u_{1}Pu_{k}\) as part of one of its blossoms. 
	 
	\textbf{Case 2.2.3.} $\text{dist}(v_{1},u_{1})\equiv 1 \pmod{2}$ y $\text{dist}(v_{1},u_{k})\equiv 0\pmod{2}$. See \cref{K4FiguraCase23}. Choosing a perfect matching of \(G\) such that the path from \(v_1\) to \(v_{2}\) is \(mm\)-alternating, and the path from \(v_{1}\) to \(v_{3}\) is \(nn\)-alternating.  We obtain a Tposy that uses \(v_{1}Pv_{2}\) as part of one of its blossoms (This blossom uses also the \(mm\)-alternating path form \(v_{4}\) to \(v_{3}\)).
	 
	 \begin{figure}[H]
	 	
	 	\begin{center}
	 		\tikzset{every picture/.style={line width=0.75pt}} 

	 		\tikzset{every picture/.style={line width=0.75pt}} 
	 		
	 		\hspace{-3cm}
	 		\begin{tikzpicture}[x=0.75pt,y=0.75pt,yscale=-1,xscale=1, scale=.85]
	 			
	 			\draw    (423.47,247.39) -- (178.37,247.39) ;
	 			\draw [shift={(178.37,247.39)}, rotate = 180] [color={rgb, 255:red, 0; green, 0; blue, 0 }  ][fill={rgb, 255:red, 0; green, 0; blue, 0 }  ][line width=0.75]      (0, 0) circle [x radius= 3.35, y radius= 3.35]   ;
	 			\draw [shift={(423.47,247.39)}, rotate = 180] [color={rgb, 255:red, 0; green, 0; blue, 0 }  ][fill={rgb, 255:red, 0; green, 0; blue, 0 }  ][line width=0.75]      (0, 0) circle [x radius= 3.35, y radius= 3.35]   ;
	 			\draw    (423.47,247.39) -- (300.92,35.13) ;
	 			\draw [shift={(300.92,35.13)}, rotate = 240] [color={rgb, 255:red, 0; green, 0; blue, 0 }  ][fill={rgb, 255:red, 0; green, 0; blue, 0 }  ][line width=0.75]      (0, 0) circle [x radius= 3.35, y radius= 3.35]   ;
	 			\draw [shift={(423.47,247.39)}, rotate = 240] [color={rgb, 255:red, 0; green, 0; blue, 0 }  ][fill={rgb, 255:red, 0; green, 0; blue, 0 }  ][line width=0.75]      (0, 0) circle [x radius= 3.35, y radius= 3.35]   ;
	 			\draw    (178.37,247.39) -- (300.92,35.13) ;
	 			\draw [shift={(300.92,35.13)}, rotate = 300] [color={rgb, 255:red, 0; green, 0; blue, 0 }  ][fill={rgb, 255:red, 0; green, 0; blue, 0 }  ][line width=0.75]      (0, 0) circle [x radius= 3.35, y radius= 3.35]   ;
	 			\draw [shift={(178.37,247.39)}, rotate = 300] [color={rgb, 255:red, 0; green, 0; blue, 0 }  ][fill={rgb, 255:red, 0; green, 0; blue, 0 }  ][line width=0.75]      (0, 0) circle [x radius= 3.35, y radius= 3.35]   ;
	 			\draw    (423.47,247.39) -- (300.92,154.07) ;
	 			\draw [shift={(300.92,154.07)}, rotate = 217.29] [color={rgb, 255:red, 0; green, 0; blue, 0 }  ][fill={rgb, 255:red, 0; green, 0; blue, 0 }  ][line width=0.75]      (0, 0) circle [x radius= 3.35, y radius= 3.35]   ;
	 			\draw [shift={(423.47,247.39)}, rotate = 217.29] [color={rgb, 255:red, 0; green, 0; blue, 0 }  ][fill={rgb, 255:red, 0; green, 0; blue, 0 }  ][line width=0.75]      (0, 0) circle [x radius= 3.35, y radius= 3.35]   ;
	 			\draw    (178.37,247.39) -- (300.92,154.07) ;
	 			\draw [shift={(300.92,154.07)}, rotate = 322.71] [color={rgb, 255:red, 0; green, 0; blue, 0 }  ][fill={rgb, 255:red, 0; green, 0; blue, 0 }  ][line width=0.75]      (0, 0) circle [x radius= 3.35, y radius= 3.35]   ;
	 			\draw [shift={(178.37,247.39)}, rotate = 322.71] [color={rgb, 255:red, 0; green, 0; blue, 0 }  ][fill={rgb, 255:red, 0; green, 0; blue, 0 }  ][line width=0.75]      (0, 0) circle [x radius= 3.35, y radius= 3.35]   ;
	 			\draw    (300.92,154.07) -- (300.92,35.13) ;
	 			\draw [shift={(300.92,35.13)}, rotate = 270] [color={rgb, 255:red, 0; green, 0; blue, 0 }  ][fill={rgb, 255:red, 0; green, 0; blue, 0 }  ][line width=0.75]      (0, 0) circle [x radius= 3.35, y radius= 3.35]   ;
	 			\draw [shift={(300.92,154.07)}, rotate = 270] [color={rgb, 255:red, 0; green, 0; blue, 0 }  ][fill={rgb, 255:red, 0; green, 0; blue, 0 }  ][line width=0.75]      (0, 0) circle [x radius= 3.35, y radius= 3.35]   ;
	 			\draw    (267.47,247.39) -- (178.37,247.39) ;
	 			\draw [shift={(178.37,247.39)}, rotate = 180] [color={rgb, 255:red, 0; green, 0; blue, 0 }  ][fill={rgb, 255:red, 0; green, 0; blue, 0 }  ][line width=0.75]      (0, 0) circle [x radius= 3.35, y radius= 3.35]   ;
	 			\draw [shift={(267.47,247.39)}, rotate = 180] [color={rgb, 255:red, 0; green, 0; blue, 0 }  ][fill={rgb, 255:red, 0; green, 0; blue, 0 }  ][line width=0.75]      (0, 0) circle [x radius= 3.35, y radius= 3.35]   ;
	 			\draw    (423.47,247.39) -- (334.37,247.39) ;
	 			\draw [shift={(334.37,247.39)}, rotate = 180] [color={rgb, 255:red, 0; green, 0; blue, 0 }  ][fill={rgb, 255:red, 0; green, 0; blue, 0 }  ][line width=0.75]      (0, 0) circle [x radius= 3.35, y radius= 3.35]   ;
	 			\draw [shift={(423.47,247.39)}, rotate = 180] [color={rgb, 255:red, 0; green, 0; blue, 0 }  ][fill={rgb, 255:red, 0; green, 0; blue, 0 }  ][line width=0.75]      (0, 0) circle [x radius= 3.35, y radius= 3.35]   ;
	 			\draw    (223.47,168.07) -- (178.37,247.39) ;
	 			\draw [shift={(178.37,247.39)}, rotate = 119.62] [color={rgb, 255:red, 0; green, 0; blue, 0 }  ][fill={rgb, 255:red, 0; green, 0; blue, 0 }  ][line width=0.75]      (0, 0) circle [x radius= 3.35, y radius= 3.35]   ;
	 			\draw [shift={(223.47,168.07)}, rotate = 119.62] [color={rgb, 255:red, 0; green, 0; blue, 0 }  ][fill={rgb, 255:red, 0; green, 0; blue, 0 }  ][line width=0.75]      (0, 0) circle [x radius= 3.35, y radius= 3.35]   ;
	 			\draw    (300.92,35.13) -- (255.82,114.45) ;
	 			\draw [shift={(255.82,114.45)}, rotate = 119.62] [color={rgb, 255:red, 0; green, 0; blue, 0 }  ][fill={rgb, 255:red, 0; green, 0; blue, 0 }  ][line width=0.75]      (0, 0) circle [x radius= 3.35, y radius= 3.35]   ;
	 			\draw [shift={(300.92,35.13)}, rotate = 119.62] [color={rgb, 255:red, 0; green, 0; blue, 0 }  ][fill={rgb, 255:red, 0; green, 0; blue, 0 }  ][line width=0.75]      (0, 0) circle [x radius= 3.35, y radius= 3.35]   ;
	 			\draw    (267.47,247.39) .. controls (120.76,399.2) and (36.76,182.2) .. (223.47,168.07) ;
	 			\draw [color={rgb, 255:red, 255; green, 0; blue, 0 }  ,draw opacity=1 ][line width=1.5]    (267.47,247.39) -- (178.37,247.39) ;
	 			\draw [color={rgb, 255:red, 255; green, 0; blue, 0 }  ,draw opacity=1 ][line width=1.5]    (255.82,114.45) -- (223.47,168.07) ;
	 			\draw [color={rgb, 255:red, 255; green, 0; blue, 0 }  ,draw opacity=1 ][line width=1.5]    (300.92,154.07) -- (300.92,35.13) ;
	 			\draw [color={rgb, 255:red, 255; green, 0; blue, 0 }  ,draw opacity=1 ][line width=1.5]    (423.47,247.39) -- (334.37,247.39) ;
	 			\draw [color={rgb, 255:red, 126; green, 211; blue, 33 }  ,draw opacity=0.35 ][line width=6]    (223.47,168.07) .. controls (239.76,139.08) and (248.76,125.08) .. (255.82,114.45) .. controls (262.88,103.82) and (297.47,33.4) .. (300.92,35.13) .. controls (304.37,36.86) and (301.76,132.08) .. (300.92,154.07) ;
	 			\draw    (267.47,247.39) ;
	 			\draw [shift={(267.47,247.39)}, rotate = 0] [color={rgb, 255:red, 0; green, 0; blue, 0 }  ][fill={rgb, 255:red, 0; green, 0; blue, 0 }  ][line width=0.75]      (0, 0) circle [x radius= 3.35, y radius= 3.35]   ;
	 			\draw [shift={(267.47,247.39)}, rotate = 0] [color={rgb, 255:red, 0; green, 0; blue, 0 }  ][fill={rgb, 255:red, 0; green, 0; blue, 0 }  ][line width=0.75]      (0, 0) circle [x radius= 3.35, y radius= 3.35]   ;
	 			\draw    (300.92,35.13) ;
	 			\draw [shift={(300.92,35.13)}, rotate = 0] [color={rgb, 255:red, 0; green, 0; blue, 0 }  ][fill={rgb, 255:red, 0; green, 0; blue, 0 }  ][line width=0.75]      (0, 0) circle [x radius= 3.35, y radius= 3.35]   ;
	 			\draw [shift={(300.92,35.13)}, rotate = 0] [color={rgb, 255:red, 0; green, 0; blue, 0 }  ][fill={rgb, 255:red, 0; green, 0; blue, 0 }  ][line width=0.75]      (0, 0) circle [x radius= 3.35, y radius= 3.35]   ;
	 			\draw    (300.92,154.07) ;
	 			\draw [shift={(300.92,154.07)}, rotate = 0] [color={rgb, 255:red, 0; green, 0; blue, 0 }  ][fill={rgb, 255:red, 0; green, 0; blue, 0 }  ][line width=0.75]      (0, 0) circle [x radius= 3.35, y radius= 3.35]   ;
	 			\draw [shift={(300.92,154.07)}, rotate = 0] [color={rgb, 255:red, 0; green, 0; blue, 0 }  ][fill={rgb, 255:red, 0; green, 0; blue, 0 }  ][line width=0.75]      (0, 0) circle [x radius= 3.35, y radius= 3.35]   ;
	 			\draw    (255.82,114.45) ;
	 			\draw [shift={(255.82,114.45)}, rotate = 0] [color={rgb, 255:red, 0; green, 0; blue, 0 }  ][fill={rgb, 255:red, 0; green, 0; blue, 0 }  ][line width=0.75]      (0, 0) circle [x radius= 3.35, y radius= 3.35]   ;
	 			\draw [shift={(255.82,114.45)}, rotate = 0] [color={rgb, 255:red, 0; green, 0; blue, 0 }  ][fill={rgb, 255:red, 0; green, 0; blue, 0 }  ][line width=0.75]      (0, 0) circle [x radius= 3.35, y radius= 3.35]   ;
	 			\draw    (223.47,168.07) ;
	 			\draw [shift={(223.47,168.07)}, rotate = 0] [color={rgb, 255:red, 0; green, 0; blue, 0 }  ][fill={rgb, 255:red, 0; green, 0; blue, 0 }  ][line width=0.75]      (0, 0) circle [x radius= 3.35, y radius= 3.35]   ;
	 			\draw [shift={(223.47,168.07)}, rotate = 0] [color={rgb, 255:red, 0; green, 0; blue, 0 }  ][fill={rgb, 255:red, 0; green, 0; blue, 0 }  ][line width=0.75]      (0, 0) circle [x radius= 3.35, y radius= 3.35]   ;
	 			\draw    (423.47,247.39) ;
	 			\draw [shift={(423.47,247.39)}, rotate = 0] [color={rgb, 255:red, 0; green, 0; blue, 0 }  ][fill={rgb, 255:red, 0; green, 0; blue, 0 }  ][line width=0.75]      (0, 0) circle [x radius= 3.35, y radius= 3.35]   ;
	 			\draw [shift={(423.47,247.39)}, rotate = 0] [color={rgb, 255:red, 0; green, 0; blue, 0 }  ][fill={rgb, 255:red, 0; green, 0; blue, 0 }  ][line width=0.75]      (0, 0) circle [x radius= 3.35, y radius= 3.35]   ;
	 			\draw    (178.37,247.39) ;
	 			\draw [shift={(178.37,247.39)}, rotate = 0] [color={rgb, 255:red, 0; green, 0; blue, 0 }  ][fill={rgb, 255:red, 0; green, 0; blue, 0 }  ][line width=0.75]      (0, 0) circle [x radius= 3.35, y radius= 3.35]   ;
	 			\draw [shift={(178.37,247.39)}, rotate = 0] [color={rgb, 255:red, 0; green, 0; blue, 0 }  ][fill={rgb, 255:red, 0; green, 0; blue, 0 }  ][line width=0.75]      (0, 0) circle [x radius= 3.35, y radius= 3.35]   ;
	 			\draw    (334.37,247.39) ;
	 			\draw [shift={(334.37,247.39)}, rotate = 0] [color={rgb, 255:red, 0; green, 0; blue, 0 }  ][fill={rgb, 255:red, 0; green, 0; blue, 0 }  ][line width=0.75]      (0, 0) circle [x radius= 3.35, y radius= 3.35]   ;
	 			\draw [shift={(334.37,247.39)}, rotate = 0] [color={rgb, 255:red, 0; green, 0; blue, 0 }  ][fill={rgb, 255:red, 0; green, 0; blue, 0 }  ][line width=0.75]      (0, 0) circle [x radius= 3.35, y radius= 3.35]   ;
	 			\draw  [color={rgb, 255:red, 189; green, 16; blue, 224 }  ,draw opacity=0.23 ][line width=6]  (178.37,247.39) .. controls (177.76,248.12) and (207.17,247.66) .. (267.47,247.39) .. controls (327.76,247.12) and (423.76,251.12) .. (423.47,247.39) .. controls (423.17,243.66) and (385.76,219.12) .. (362.19,200.73) .. controls (338.62,182.34) and (303.76,153.12) .. (300.92,154.07) .. controls (298.08,155.01) and (257.53,188.34) .. (239.64,200.73) .. controls (221.76,213.12) and (178.98,246.66) .. (178.37,247.39) -- cycle ;
	 			\draw  [color={rgb, 255:red, 80; green, 227; blue, 194 }  ,draw opacity=0.27 ][line width=6]  (159.76,185.12) .. controls (133.76,197.12) and (110.76,229.12) .. (115.76,254.12) .. controls (120.76,279.12) and (153.76,311.12) .. (184.76,300.12) .. controls (215.76,289.12) and (270.86,250.3) .. (267.47,247.39) .. controls (264.08,244.48) and (182.98,252.66) .. (178.37,247.39) .. controls (173.76,242.12) and (232.17,168.01) .. (223.47,168.07) .. controls (214.76,168.12) and (185.76,173.12) .. (159.76,185.12) -- cycle ;
	 			
	 			\draw (154,242.4) node [anchor=north west][inner sep=0.75pt]    {$v_{1}$};
	 			\draw (293,12.4) node [anchor=north west][inner sep=0.75pt]    {$v_{2}$};
	 			\draw (428,244.4) node [anchor=north west][inner sep=0.75pt]    {$v_{3}$};
	 			\draw (306,137.4) node [anchor=north west][inner sep=0.75pt]    {$v_{4}$};
	 			
	 			\draw (210,145) node [anchor=north][inner sep=0.75pt]    {$u_{1}$};
	 			\draw (275,275) node [anchor=south][inner sep=0.75pt]    {$u_{k}$};
	 			\draw (140,320) node [anchor= west][inner sep=0.75pt]    {$u_{1}Pu_{k}$};
	 		\end{tikzpicture}

	 		\vspace{-2cm}
	 	\end{center}
	 	\caption{Even subdivision of \(K_{4}\). Case 2.2}
	 	\label{K4FiguraCase22}
	 	
	 \end{figure}
	 
	 \begin{figure}[H]
	 	
	 	\begin{center}

	 		\tikzset{every picture/.style={line width=0.75pt}} 
	 		
	 		\hspace{-3cm}
	 		\begin{tikzpicture}[x=0.75pt,y=0.75pt,yscale=-1,xscale=1, scale=.85]
	 			
	 			\draw    (430.47,246.39) -- (185.37,246.39) ;
	 			\draw [shift={(185.37,246.39)}, rotate = 180] [color={rgb, 255:red, 0; green, 0; blue, 0 }  ][fill={rgb, 255:red, 0; green, 0; blue, 0 }  ][line width=0.75]      (0, 0) circle [x radius= 3.35, y radius= 3.35]   ;
	 			\draw [shift={(430.47,246.39)}, rotate = 180] [color={rgb, 255:red, 0; green, 0; blue, 0 }  ][fill={rgb, 255:red, 0; green, 0; blue, 0 }  ][line width=0.75]      (0, 0) circle [x radius= 3.35, y radius= 3.35]   ;
	 			\draw    (430.47,246.39) -- (307.92,34.13) ;
	 			\draw [shift={(307.92,34.13)}, rotate = 240] [color={rgb, 255:red, 0; green, 0; blue, 0 }  ][fill={rgb, 255:red, 0; green, 0; blue, 0 }  ][line width=0.75]      (0, 0) circle [x radius= 3.35, y radius= 3.35]   ;
	 			\draw [shift={(430.47,246.39)}, rotate = 240] [color={rgb, 255:red, 0; green, 0; blue, 0 }  ][fill={rgb, 255:red, 0; green, 0; blue, 0 }  ][line width=0.75]      (0, 0) circle [x radius= 3.35, y radius= 3.35]   ;
	 			\draw    (185.37,246.39) -- (307.92,34.13) ;
	 			\draw [shift={(307.92,34.13)}, rotate = 300] [color={rgb, 255:red, 0; green, 0; blue, 0 }  ][fill={rgb, 255:red, 0; green, 0; blue, 0 }  ][line width=0.75]      (0, 0) circle [x radius= 3.35, y radius= 3.35]   ;
	 			\draw [shift={(185.37,246.39)}, rotate = 300] [color={rgb, 255:red, 0; green, 0; blue, 0 }  ][fill={rgb, 255:red, 0; green, 0; blue, 0 }  ][line width=0.75]      (0, 0) circle [x radius= 3.35, y radius= 3.35]   ;
	 			\draw    (430.47,246.39) -- (307.92,153.07) ;
	 			\draw [shift={(307.92,153.07)}, rotate = 217.29] [color={rgb, 255:red, 0; green, 0; blue, 0 }  ][fill={rgb, 255:red, 0; green, 0; blue, 0 }  ][line width=0.75]      (0, 0) circle [x radius= 3.35, y radius= 3.35]   ;
	 			\draw [shift={(430.47,246.39)}, rotate = 217.29] [color={rgb, 255:red, 0; green, 0; blue, 0 }  ][fill={rgb, 255:red, 0; green, 0; blue, 0 }  ][line width=0.75]      (0, 0) circle [x radius= 3.35, y radius= 3.35]   ;
	 			\draw    (185.37,246.39) -- (307.92,153.07) ;
	 			\draw [shift={(307.92,153.07)}, rotate = 322.71] [color={rgb, 255:red, 0; green, 0; blue, 0 }  ][fill={rgb, 255:red, 0; green, 0; blue, 0 }  ][line width=0.75]      (0, 0) circle [x radius= 3.35, y radius= 3.35]   ;
	 			\draw [shift={(185.37,246.39)}, rotate = 322.71] [color={rgb, 255:red, 0; green, 0; blue, 0 }  ][fill={rgb, 255:red, 0; green, 0; blue, 0 }  ][line width=0.75]      (0, 0) circle [x radius= 3.35, y radius= 3.35]   ;
	 			\draw    (307.92,153.07) -- (307.92,34.13) ;
	 			\draw [shift={(307.92,34.13)}, rotate = 270] [color={rgb, 255:red, 0; green, 0; blue, 0 }  ][fill={rgb, 255:red, 0; green, 0; blue, 0 }  ][line width=0.75]      (0, 0) circle [x radius= 3.35, y radius= 3.35]   ;
	 			\draw [shift={(307.92,153.07)}, rotate = 270] [color={rgb, 255:red, 0; green, 0; blue, 0 }  ][fill={rgb, 255:red, 0; green, 0; blue, 0 }  ][line width=0.75]      (0, 0) circle [x radius= 3.35, y radius= 3.35]   ;
	 			\draw    (274.47,246.39) -- (185.37,246.39) ;
	 			\draw [shift={(185.37,246.39)}, rotate = 180] [color={rgb, 255:red, 0; green, 0; blue, 0 }  ][fill={rgb, 255:red, 0; green, 0; blue, 0 }  ][line width=0.75]      (0, 0) circle [x radius= 3.35, y radius= 3.35]   ;
	 			\draw [shift={(274.47,246.39)}, rotate = 180] [color={rgb, 255:red, 0; green, 0; blue, 0 }  ][fill={rgb, 255:red, 0; green, 0; blue, 0 }  ][line width=0.75]      (0, 0) circle [x radius= 3.35, y radius= 3.35]   ;
	 			\draw    (430.47,246.39) -- (341.37,246.39) ;
	 			\draw [shift={(341.37,246.39)}, rotate = 180] [color={rgb, 255:red, 0; green, 0; blue, 0 }  ][fill={rgb, 255:red, 0; green, 0; blue, 0 }  ][line width=0.75]      (0, 0) circle [x radius= 3.35, y radius= 3.35]   ;
	 			\draw [shift={(430.47,246.39)}, rotate = 180] [color={rgb, 255:red, 0; green, 0; blue, 0 }  ][fill={rgb, 255:red, 0; green, 0; blue, 0 }  ][line width=0.75]      (0, 0) circle [x radius= 3.35, y radius= 3.35]   ;
	 			\draw    (230.47,167.07) -- (185.37,246.39) ;
	 			\draw [shift={(185.37,246.39)}, rotate = 119.62] [color={rgb, 255:red, 0; green, 0; blue, 0 }  ][fill={rgb, 255:red, 0; green, 0; blue, 0 }  ][line width=0.75]      (0, 0) circle [x radius= 3.35, y radius= 3.35]   ;
	 			\draw [shift={(230.47,167.07)}, rotate = 119.62] [color={rgb, 255:red, 0; green, 0; blue, 0 }  ][fill={rgb, 255:red, 0; green, 0; blue, 0 }  ][line width=0.75]      (0, 0) circle [x radius= 3.35, y radius= 3.35]   ;
	 			\draw    (307.92,34.13) -- (262.82,113.45) ;
	 			\draw [shift={(262.82,113.45)}, rotate = 119.62] [color={rgb, 255:red, 0; green, 0; blue, 0 }  ][fill={rgb, 255:red, 0; green, 0; blue, 0 }  ][line width=0.75]      (0, 0) circle [x radius= 3.35, y radius= 3.35]   ;
	 			\draw [shift={(307.92,34.13)}, rotate = 119.62] [color={rgb, 255:red, 0; green, 0; blue, 0 }  ][fill={rgb, 255:red, 0; green, 0; blue, 0 }  ][line width=0.75]      (0, 0) circle [x radius= 3.35, y radius= 3.35]   ;
	 			\draw    (274.47,246.39) .. controls (77.76,395.71) and (76.11,127.58) .. (262.82,113.45) ;
	 			\draw [color={rgb, 255:red, 255; green, 0; blue, 0 }  ,draw opacity=1 ][line width=1.5]    (274.47,246.39) -- (341.37,246.39) ;
	 			\draw [color={rgb, 255:red, 255; green, 0; blue, 0 }  ,draw opacity=1 ][line width=1.5]    (307.92,34.13) -- (262.82,113.45) ;
	 			\draw [color={rgb, 255:red, 255; green, 0; blue, 0 }  ,draw opacity=1 ][line width=1.5]    (230.47,167.07) -- (185.37,246.39) ;
	 			\draw [color={rgb, 255:red, 255; green, 0; blue, 0 }  ,draw opacity=1 ][line width=1.5]    (307.92,153.07) -- (430.47,246.39) ;
	 			\draw    (185.37,246.39) ;
	 			\draw [shift={(185.37,246.39)}, rotate = 0] [color={rgb, 255:red, 0; green, 0; blue, 0 }  ][fill={rgb, 255:red, 0; green, 0; blue, 0 }  ][line width=0.75]      (0, 0) circle [x radius= 3.35, y radius= 3.35]   ;
	 			\draw [shift={(185.37,246.39)}, rotate = 0] [color={rgb, 255:red, 0; green, 0; blue, 0 }  ][fill={rgb, 255:red, 0; green, 0; blue, 0 }  ][line width=0.75]      (0, 0) circle [x radius= 3.35, y radius= 3.35]   ;
	 			\draw    (307.92,153.07) ;
	 			\draw [shift={(307.92,153.07)}, rotate = 0] [color={rgb, 255:red, 0; green, 0; blue, 0 }  ][fill={rgb, 255:red, 0; green, 0; blue, 0 }  ][line width=0.75]      (0, 0) circle [x radius= 3.35, y radius= 3.35]   ;
	 			\draw [shift={(307.92,153.07)}, rotate = 0] [color={rgb, 255:red, 0; green, 0; blue, 0 }  ][fill={rgb, 255:red, 0; green, 0; blue, 0 }  ][line width=0.75]      (0, 0) circle [x radius= 3.35, y radius= 3.35]   ;
	 			\draw    (341.37,246.39) ;
	 			\draw [shift={(341.37,246.39)}, rotate = 0] [color={rgb, 255:red, 0; green, 0; blue, 0 }  ][fill={rgb, 255:red, 0; green, 0; blue, 0 }  ][line width=0.75]      (0, 0) circle [x radius= 3.35, y radius= 3.35]   ;
	 			\draw [shift={(341.37,246.39)}, rotate = 0] [color={rgb, 255:red, 0; green, 0; blue, 0 }  ][fill={rgb, 255:red, 0; green, 0; blue, 0 }  ][line width=0.75]      (0, 0) circle [x radius= 3.35, y radius= 3.35]   ;
	 			\draw    (274.47,246.39) ;
	 			\draw [shift={(274.47,246.39)}, rotate = 0] [color={rgb, 255:red, 0; green, 0; blue, 0 }  ][fill={rgb, 255:red, 0; green, 0; blue, 0 }  ][line width=0.75]      (0, 0) circle [x radius= 3.35, y radius= 3.35]   ;
	 			\draw [shift={(274.47,246.39)}, rotate = 0] [color={rgb, 255:red, 0; green, 0; blue, 0 }  ][fill={rgb, 255:red, 0; green, 0; blue, 0 }  ][line width=0.75]      (0, 0) circle [x radius= 3.35, y radius= 3.35]   ;
	 			\draw    (307.92,34.13) ;
	 			\draw [shift={(307.92,34.13)}, rotate = 0] [color={rgb, 255:red, 0; green, 0; blue, 0 }  ][fill={rgb, 255:red, 0; green, 0; blue, 0 }  ][line width=0.75]      (0, 0) circle [x radius= 3.35, y radius= 3.35]   ;
	 			\draw [shift={(307.92,34.13)}, rotate = 0] [color={rgb, 255:red, 0; green, 0; blue, 0 }  ][fill={rgb, 255:red, 0; green, 0; blue, 0 }  ][line width=0.75]      (0, 0) circle [x radius= 3.35, y radius= 3.35]   ;
	 			\draw    (262.82,113.45) ;
	 			\draw [shift={(262.82,113.45)}, rotate = 0] [color={rgb, 255:red, 0; green, 0; blue, 0 }  ][fill={rgb, 255:red, 0; green, 0; blue, 0 }  ][line width=0.75]      (0, 0) circle [x radius= 3.35, y radius= 3.35]   ;
	 			\draw [shift={(262.82,113.45)}, rotate = 0] [color={rgb, 255:red, 0; green, 0; blue, 0 }  ][fill={rgb, 255:red, 0; green, 0; blue, 0 }  ][line width=0.75]      (0, 0) circle [x radius= 3.35, y radius= 3.35]   ;
	 			\draw    (230.47,167.07) ;
	 			\draw [shift={(230.47,167.07)}, rotate = 0] [color={rgb, 255:red, 0; green, 0; blue, 0 }  ][fill={rgb, 255:red, 0; green, 0; blue, 0 }  ][line width=0.75]      (0, 0) circle [x radius= 3.35, y radius= 3.35]   ;
	 			\draw [shift={(230.47,167.07)}, rotate = 0] [color={rgb, 255:red, 0; green, 0; blue, 0 }  ][fill={rgb, 255:red, 0; green, 0; blue, 0 }  ][line width=0.75]      (0, 0) circle [x radius= 3.35, y radius= 3.35]   ;
	 			\draw    (430.47,246.39) ;
	 			\draw [shift={(430.47,246.39)}, rotate = 0] [color={rgb, 255:red, 0; green, 0; blue, 0 }  ][fill={rgb, 255:red, 0; green, 0; blue, 0 }  ][line width=0.75]      (0, 0) circle [x radius= 3.35, y radius= 3.35]   ;
	 			\draw [shift={(430.47,246.39)}, rotate = 0] [color={rgb, 255:red, 0; green, 0; blue, 0 }  ][fill={rgb, 255:red, 0; green, 0; blue, 0 }  ][line width=0.75]      (0, 0) circle [x radius= 3.35, y radius= 3.35]   ;
	 			\draw [color={rgb, 255:red, 126; green, 211; blue, 33 }  ,draw opacity=0.35 ][line width=6]    (262.82,113.45) .. controls (283.47,76.07) and (289.47,67.07) .. (307.92,34.13) ;
	 			\draw  [color={rgb, 255:red, 189; green, 16; blue, 224 }  ,draw opacity=0.23 ][line width=6]  (430.47,246.39) .. controls (436.33,243.46) and (398.78,192.14) .. (369.19,140.26) .. controls (339.6,88.38) and (311.08,27.37) .. (307.92,34.13) .. controls (304.76,40.89) and (308.76,68.89) .. (308.76,99.89) .. controls (308.76,130.89) and (307.42,147.87) .. (307.92,153.07) .. controls (308.41,158.27) and (336.76,175.89) .. (361.76,193.89) .. controls (386.76,211.89) and (424.6,249.32) .. (430.47,246.39) -- cycle ;
	 			\draw  [color={rgb, 255:red, 80; green, 227; blue, 194 }  ,draw opacity=0.27 ][line width=6]  (131.76,201.71) .. controls (124.14,212.68) and (116.76,257.71) .. (140.76,280.71) .. controls (164.76,303.71) and (208.76,284.71) .. (224.76,279.71) .. controls (240.76,274.71) and (270.2,248.54) .. (274.47,246.39) .. controls (278.73,244.24) and (307.35,244.24) .. (317.76,245.71) .. controls (328.17,247.18) and (337.33,247.81) .. (341.37,246.39) .. controls (345.4,244.97) and (436.17,247.07) .. (430.47,246.39) .. controls (424.76,245.71) and (314.76,155.71) .. (307.92,153.07) .. controls (301.08,150.42) and (188.76,248.71) .. (185.37,246.39) .. controls (181.98,244.07) and (206.74,205.4) .. (230.47,167.07) .. controls (254.2,128.74) and (265.47,114.48) .. (262.82,113.45) .. controls (260.17,112.42) and (220.79,119.52) .. (186.76,138.71) .. controls (152.73,157.9) and (139.38,190.74) .. (131.76,201.71) -- cycle ;
	 			
	 			\draw (161,240.4) node [anchor=north west][inner sep=0.75pt]    {$v_{1}$};
	 			\draw (300,11.4) node [anchor=north west][inner sep=0.75pt]    {$v_{2}$};
	 			\draw (435,243.4) node [anchor=north west][inner sep=0.75pt]    {$v_{3}$};
	 			\draw (313,136.4) node [anchor=north west][inner sep=0.75pt]    {$v_{4}$};
	 			
	 			\draw (255,90) node [anchor=north][inner sep=0.75pt]    {$u_{1}$};
	 			\draw (275,275) node [anchor=south][inner sep=0.75pt]    {$u_{k}$};
	 			\draw (150,120) node [anchor= west][inner sep=0.75pt]    {$u_{1}Pu_{k}$};
	 			\end{tikzpicture}

	 			\vspace{-2.5cm}
 			\end{center}
 		\caption{Even subdivision of \(K_{4}\). Case 2.3}
 	\label{K4FiguraCase23}
	 	
	 \end{figure}
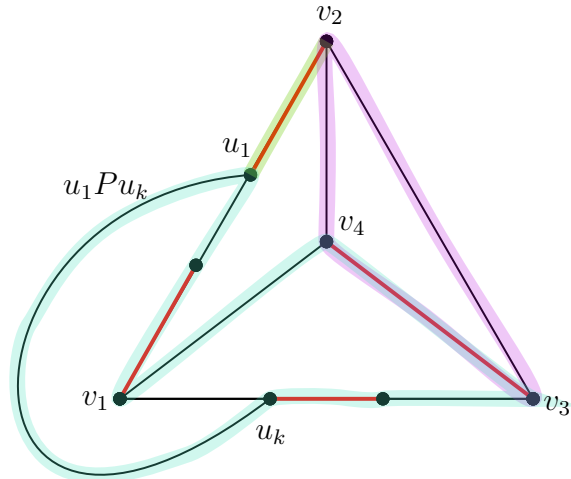

	 \textbf{Case 3.} $u_{1}\in V\left(v_{1}Pv_{2}\right)$ and $u_{k}\in V\left(v_{3}Pv_{4}\right)$. See \cref{K4FiguraCase3}. Choosing the perfect matching of \(G\) such that both path, \(v_{1}Pv_{2}\) and \(v_{3}Pv_{4}\), are \(mm\)-alternating.  We obtain a Tposy that uses \(u_{1}Pu_{k}\) as part of one of its blossoms (This blossom uses also the \(nn\)-alternating path form \(v_{2}\) to \(V_{4}\)).

	 \begin{figure}[H]
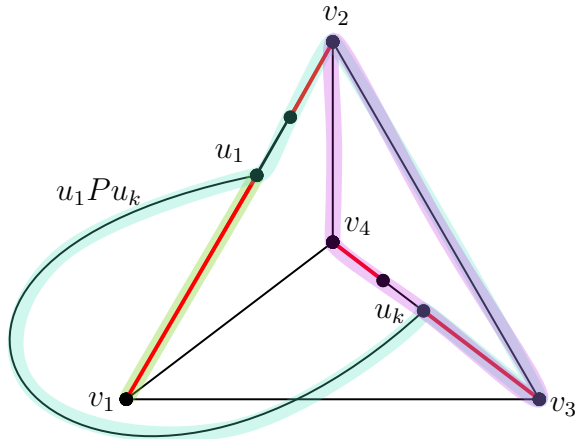

	 	
	 	\begin{center}

	 		\tikzset{every picture/.style={line width=0.75pt}} 
	 		
	 		\hspace{-4cm}


	 			\vspace{-3cm}
 			\end{center}
 		\caption{Even subdivision of \(K_{4}\). Case 3}
 	\label{K4FiguraCase3}
	 \end{figure}

\end{proof}	

%
%
%

\begin{figure}[h]
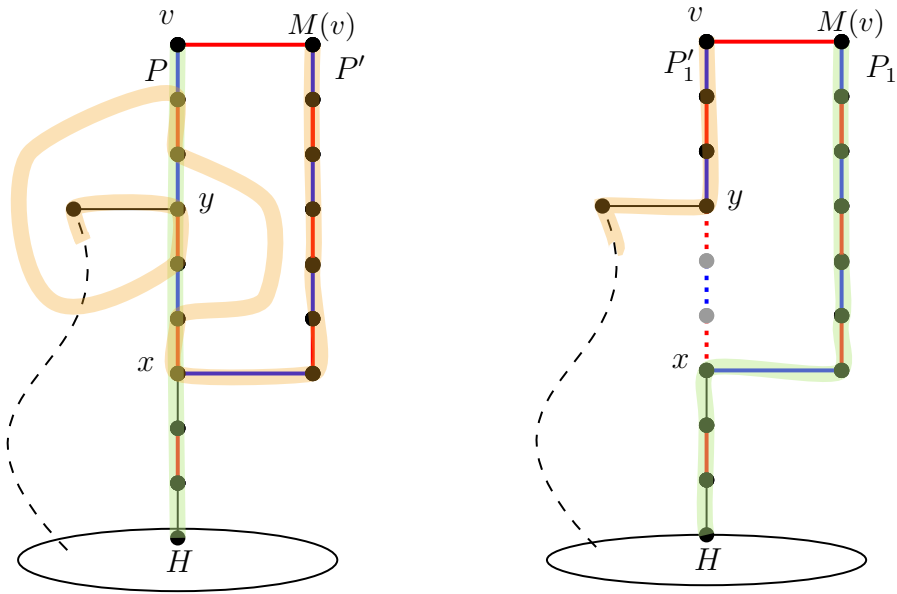

	\begin{center}			
		\tikzset{every picture/.style={line width=0.75pt}} 
		
		\hspace{-1.5cm}%
		
	\end{center}
	\caption{In this case, the matching is rotated along the cycle, yielding a new matching. Note that this does not affect the matching in $B$ or in $H$, and moreover, the resulting paths are alternating}
	\label{Figura261}
	
\end{figure}
%
%
%

The following result may appear na\"{\i}ve, but its proof is nontrivial.

\begin{theorem}\label{teo_main1}
	Let \( G \) be a Jposy-graph. Then every vertex of \( G \) belongs to an \( M \)-Tposy, for some maximum matching \( M \) of \( G \).
\end{theorem}
\begin{proof}
	The proof is by induction on the size of the graph $G$. The base case is direct. By \cref{lem:noKE1}, $G$ is not a K\H{o}nig-Egerv\'ary graph, and by \cref{safe}, $G$ contains an $M$-Tposy $H$. Given $v\in V(G)- V(H)$, we will show that $v$ belongs to an Tposy of $G$ relative a some maximum matching \(M^{*}\) of \(G\).
		
	For any vertex \( x \in V(G) \setminus V(H) \), by \cref{triviallemma}, there exists an \( M \)-\( nn \)-alternating walk \(W\) from \( x \) to a vertex in \( H \). Without loss of generality, assume that such an \( M \)-\( nn \)-alternating walk visits exactly one vertex \(h\) of \( H \), and does so only once. We call this vertex \(H\) the ``entry'' to \(H\).
	By \cref{lem:JObs}, from there exists an \(M\)-\(mm\)-alternating path from the entry to the base of some of the blossoms of \(H\).  Therefore, for any vertex \( x \in V(G) \setminus V(H) \), by \cref{adfw}, one of the following statements hold:
		\begin{enumerate}
			\item\label{Con_1} there exists an $nn$-alternating subpath of \(W\) from $x$ to a vertex of \(H\) internally disjoint from \(H\),
			\item\label{Con_2} there exists an $nm$-alternating subpath $P = x, \dots, b$ of \(W\) connecting $x$ to the base of a \(M\)-blossom $B$ with base \(b\) such that $\left(V(P)\cup V(B)\right)\cap V(H) = \emptyset$ and $V(P)\cap V(B) = \{b\}$.
		\end{enumerate}
		Using this observation for either $v$ or $M(v)$, we distinguish the following three cases:
		
		\textbf{Case 1.} Suppose both $v$ and $M(v)$ satisfy \cref{Con_2}. Then $v$ is contained in a Jposy disjoint from $H$. By induction, there exists a Tposy of $G$ containing $v$.
		
		\textbf{Case 2.} Suppose both $v$ and $M(v)$ satisfy \cref{Con_1}. Let $P$ and $P'$ be the paths associated with $v$ and $M(v)$ respectively, as guaranteed by \cref{Con_1}. We consider two subcases.
		
		\textbf{Case 2.1.} $V(P)\cap V(P') = \emptyset$. This case follows directly from \cref{Lem_T}.
		
		\textbf{Case 2.2.} $V(P)\cap V(P') \neq \emptyset$. If $P'$ uses an edge $yM(y)\in E(P)$ in the same direction as $P$, where \(M(y)\) is closer to \(H\) than \(y\), then it emerges an \(M\)-Jposy given by the \(M\)-blossom \(y,Pv,M(v),P^\prime y\), the \(M\)-\(nn\)-alternating path \(y,Ph,M(h),\dots, b\) the base of one of the \(M\)-blossom in \(H\) and the this \(M\)-blossom. Thus, by induction there exists a $T$-posy of $G$ containing $v$.
		
		Assume that all the edges in common between \(P\) and \(P\prime\) are use in opposite direction. Let $x$ be the first vertex of $P'$ lying in $V(P)$. If $P$ forms an odd cycle with $P'x$, then $v$ is contained in a blossom whose base is $mm$-connected to a blossom base in $H$. Thus, by induction there exists a $T$-posy of $G$ containing $v$.
		
		Suppose the cycle is even. Let $y$ be the last vertex of $P'$ in $Px$. Then we redefine the paths $P$ and $P'$ by rotating the matching in the cycle, call this new maximum matching \(M^\prime\), see \cref{Figura261}. Note that \(H\) is also Tposy relative to \(M^\prime\). We define:
		\[
		P_1 = M(v), P'x, xP 
		\]
		and 
		\[
		P_1' = v, Py, yP'.
		\]
		These paths and the two ``half Tposy'' of \(H\) that they determine form an \(M^\prime \)-Jposy with less edges. Thus, by induction there exists a $T$-posy of $G$ containing $v$.
		
		\textbf{Case 3.} Suppose $v$ does not satisfy \cref{Con_1} but $M(v)$ does. Let $P$ and $B$ be the path and blossom obtained by $v$ satisfying \cref{Con_2}, and let $P'$ be the path for $M(v)$ to \(H\) given by \cref{Con_1}. If we consider the ``half Tposy'' from the entry of \(P'\) to \(H\), it emerges a an smaller Jposy relative to \(M\). Thus, by induction there exists a $T$-posy of $G$ containing $v$.
	\end{proof}

	%
	%
	%

		As an immediate corollary of the previous theorem, we obtain the next result.
	
	\begin{corollary}
		Let $G$ be a graph with a perfect matching. $G$ has the K\H{o}nig-Egerv\'ary property if and only if it has no an Jposy relative to any perfect matching $M$.
	\end{corollary}	
	\begin{proof}
		Let $H$ be an $M$-Jposy of $G$ and let $v \in V(G)$. Then both $G - V(H)$ and $H$ have perfect matchings. By \cref{teo_main1}, there exists an $M'$-Tposy $H_1$ of $H$ such that $v \in V(H_1)$, where $M'$ is a perfect matching of $H$. Then 
		\[
		M'' = M' \cup \left(M \setminus E(H)\right)
		\]
		is a perfect matching of $G$, and there exists an $M''$-Tposy of $G$ containing $v$. The result follow by \cref{safe}.
	\end{proof}

	%
	%
	%
	
	\section{Jflower vertices are Jposy vertices or flower vertices }\label{Flower}
	
	\begin{theorem}\label{123d}
		If $G$ is an Jflower graph, then every $v\in V(G)$ is in an flower or in an Tposy of $G$ relative to some maximum matching.
	\end{theorem}
	\begin{proof}
		Let $u$ be the unsaturated vertex and let $G^{\prime}$ be the graph obtained from $G$ by attaching to $u$ a triangle connected by an edge,
		that is, $G^{\prime}=G^{\prime}(V(G)\cup\{x,y,z,w\},E(G)\cup\{ux,xy,xz,yz\})$; see \cref{Figuraflower}. Then $G^{\prime}$ is a Jposy graph. Note that the edges
		$ux,xy,xz,yz$ are contained in every perfect matching of $G^{\prime}$. Let $v\in V(G)$ and let $H$ be an $M$-Tposy of $G^{\prime}$ that contains $v$. If $V(H)\subseteq V(G)$, then $H$ is an $M^{\prime}\cap E(G)$-posy of $G$, where $M^{\prime}=M-\{ux,xy,xz,yz\}$. Otherwise, since a Tposy has no vertices of degree one, we must have $x,y,z,w\in V(H)$. Therefore, the vertices $V(H)-\{x,y,z,w\}$ induce an $M^{\prime}$-flower of $G$ that contains $v$.
	\end{proof}

\begin{figure}[h]
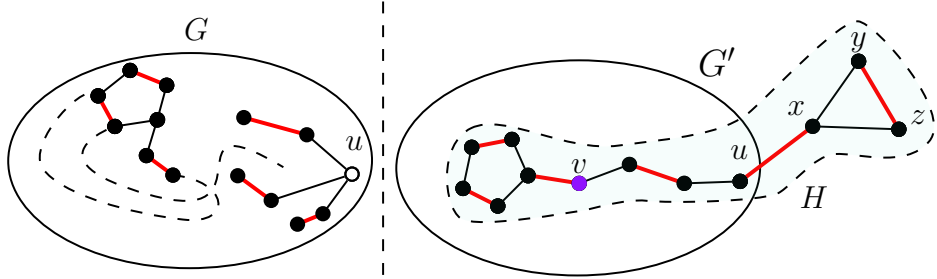

	\begin{center}			
		\tikzset{every picture/.style={line width=0.75pt}} 

	\end{center}
	\caption{Illustration of the proof of \cref{123d}}
	\label{Figuraflower}
\end{figure}

	%
	%
	%
	
	\section{Main result}
	
	Let \( G \) be a graph. Define \( V_{ESG}(G) \) as the set of all vertices of \( G \) that belong to some flower or some posy of \( G \). Similarly, let \( V_{T}(G) \) be the set of vertices of \( G \) that are contained in a flower or a Tposy, and let \( V_{J}(G) \) be the set of vertices of \( G \) that appear in some Jflower or some Jposy. Although Jflower and Jposy are not graph, it is natural to talk about vertices and edges of this configurations.
	
	Clearly, for any graph \( G \), we have the inclusions
	\[
	V_{T}(G) \subset V_{ESG}(G) \subset V_{J}(G).
	\]

		The main objective of this work is to establish the following result:
	
	\begin{theorem} \label{teo_main}
		For any graph \( G \), the following equalities hold:
		\[
		V_{T}(G) = V_{ESG}(G) = V_{J}(G).
		\]
	\end{theorem}
	
	The vertices in this set are called the \emph{Sterboul-Deming-vertices}, or just \emph{SD-vertices}, of \( G \), and from this point on, the set of SD-vertices will be denoted by \( V_{SD}(G) \). A graph is called an \emph{Sterboul Deming graph} if \( V(G) = V_{SD}(G) \), Molina and Jaume originally called them FP-graphs. 
	
	\begin{proof}
		Let $v\in V_{SD}(G)$. If $v$ is in an $M$-Jposy $H$ of $G$, then by \cref{teo_main1}, we can modify $M$ within $H$ and find an $M^{\prime}$-Tposy of $G$ that contains $v$. Similarly, if $v$ is in an $M$-Jflower $H$ of $G$, then by \cref{123d}, we can modify the matching within $H$ and find an $M^{\prime}$-Tposy or an $M^{\prime}$-flower of $G$ that contains $v$.
	\end{proof}

	Sterboul-Deming graphs are related to the Hamiltonian cycle problem. Clearly, every Hamiltonian graph with an odd number of vertices is an Sterboul-Deming graph. We conjecture that every Hamiltonian graph with an even number of vertices is either a K\H{o}nig–Egerv\'ary graph or an Sterboul-Deming graph.

	%
	%
	%

	%
	%
	%
	\section*{Acknowledgments}
	
	This work was partially supported by Universidad Nacional de San Luis (Argentina), PROICO 03-0723, MATH AmSud, grant 22-MATH-02, Agencia I+D+i (Argentina), grants PICT-2020-Serie A-00549 and PICT-2021-CAT-II-00105, CONICET (Argentina) grant PIP 11220220100068CO.
	%
	%
	%
	
	
	\section*{Declaration of generative AI and AI-assisted technologies in the writing process}
	During the preparation of this work the authors used ChatGPT-3.5 in order to improve the grammar of several paragraphs of the text. After using this service, the authors reviewed and edited the content as needed and take full responsibility for the content of the publication.
	%
	%
	%
	
	\section*{Data availability}
	
	Data sharing not applicable to this article as no datasets were generated or analyzed during the current study.
	%
	%
	%
	
	\section*{Declarations}
	
	\noindent\textbf{Conflict of interest} \ The authors declare that they have no conflict of interest.

	%
	%
	%
	
	\bibliographystyle{acm}
	
	\bibliography{TAGcitas_J_posy_Flowers}
\end{document}